\documentclass[amstex,12pt, amssymb]{article}

\usepackage{mathtext}
\usepackage[cp1251]{inputenc}
\usepackage[T2A]{fontenc}
\usepackage[dvips]{graphicx}
\usepackage{amsmath}
\usepackage{amssymb}
\usepackage{amsxtra}
\usepackage{latexsym}
\usepackage{ifthen}

\newcounter{lemma}[section]

\newcounter{corollary}[section]

\newcounter{remark}[section]

\newcounter{theorem}[section]

\newcounter{proposition}[section]

\newcounter{example}

\numberwithin{equation}{section}

\begin{document}

\markboth{\centerline{E.~SEVOST'YANOV}} {\centerline{ON THE LOCAL
AND BOUNDARY BEHAVIOR ... }}

\def\cc{\setcounter{equation}{0}
\setcounter{figure}{0}\setcounter{table}{0}}

\overfullrule=0pt


\author{{EVGENY SEVOST'YANOV}\\}

\title{
{\bf ON THE LOCAL AND BOUNDARY BEHAVIOR OF MAPPINGS ON
FACTOR-SPACES}}

\date{\today}
\maketitle

\begin{abstract}
In this article, we study mappings acting between domains of two
factor spaces by certain groups of M\"{o}bius automorphisms of the
unit ball that act discontinuously and do not have fixed points. For
such mappings, we have established estimates for the distortion of
the modulus of families of paths, which are similar to the
well-known Poletsky and V\"{a}is\"{a}l\"{a} inequalities. As
applications, we have obtained several important results on the
local and boundary behavior of mappings.
\end{abstract}


\section{Introduction} This article is devoted to the study of mappings acting between the spaces of the orbits of elements.
It should be noted that the study of spaces of this kind has a
certain significance. First of all, we mean the Poincar\'{e} theorem
on uniformization, according to which each Riemannian surface is
conformally equivalent to a certain factor-space of a flat domain
with respect to the group of linear fractional mappings. If such a
surface has a hyperblochic type, then the role of such a domain is
played by the unit disk, and the corresponding group is a family of
fractional-linear conformal automorphisms that have no fixed points
and act discontinuously inside the disk. In order not to be limited
to only the two-dimensional case, we will consider in this article
similar mappings related to an arbitrary dimension $n\geqslant 2.$
In this case, we are dealing with the study of factor spaces with
respect to a certain group of M\"{o}bius transformations of the unit
ball onto itself.

We can also point out not so numerous, but, at the same time, highly
effective attempts to study this object. On this occasion, we
mention the classic results of Martio-Srebro, as well as the recent
works of Ryazanov-Volkov, see ~\cite{MS$_1$}-\cite{MS$_2$}
and~\cite{RV$_1$}-\cite{RV$_2$}. In the first two papers, modular
inequalities are a tool for studying the boundary behavior of
mappings defined in the whole space, while the boundary element of
space is defined in some special way. In another pair of papers, on
the contrary, the boundary extension of mappings is studied with
respect to a domain entirely lying in a given space.

\medskip
The main idea of this article is to establish modular inequalities
on orbit spaces, and with their help to study the local and boundary
behavior of maps. The main difference with respect to the
above-mentioned papers is that we study maps with branching of
arbitrary dimension, which are defined only in a certain domain and
can have an unbounded quasiconformality coefficient. In addition,
even in the two-dimensional case, the estimates of the distortion of
the modulus obtained in this article are more significant compared
with~\cite{RV$_1$}-\cite{RV$_2$}.

\medskip
We now give all the necessary definitions. Let $G$ be some group of
M\"{o}bius transformations of the unit ball ${\Bbb B}^n$ onto
itself. In what follows, the points $x$ and $y\in{\Bbb B}^n$ will be
called {\it $G$-equivalent} (or shorter, {\it equivalent}), if there
is $A\in G$ such such that $x=A(y).$ A set consisting of equivalence
classes of elements according to the indicated principle is denoted
by ${\Bbb B}^n/G.$ Denote by ${\cal{G M}}({\Bbb B}^n)$ the group of
all M\"{o}bius transformations of ${\Bbb B}^n$ onto itself.
According to~\cite[Section~3.4]{MS$_1$}, the {\it hyperbolic
measure} of the Lebesgue measurable set $A\subset{\Bbb B}^n$ is
determined by the relation
\begin{equation}\label{eq2B}
V(A)=\int\limits_A\frac{2^n\, dm(x)}{{(1-|x|^2)}^n}\,.
\end{equation}
We define the {\it hyperbolic distance} $h(x, y)$ between the points
$x, y\in{\Bbb B}^n$ by the relation
\begin{equation}\label{eq3}
h(x, y)=\log\,\frac{1+t}{1-t}\,,\quad
t=\frac{|x-y|}{\sqrt{|x-y|^2+(1-|x|^2)(1-|y|^2)}}\,,
\end{equation}
see, for example, \cite[relation~(2.18), Remark~2.12 and
Exercise~2.52]{Vu}. Note that $h(x, y)=h(g(x), g(y))$ for any $g\in
{\cal{G M}}({\Bbb B}^n),$ see~\cite[relation~(2.20)]{Vu}.
In what follows, we denote by $I$ the identity mapping in~${\Bbb
R}^n.$ According to~\cite[Section~3.4]{MS$_1$}, a set of the form
\begin{equation}\label{eq1}
P=\{p\in {\Bbb B}^n: h(p, p_0)<h(p, T(p_0))\}\,,\qquad T\in
G\setminus\{I\}\,,
\end{equation}
is called a {\it normal fundamental polyhedron} with center at the
point $p_0.$ Let $\pi:{\Bbb B}^n\rightarrow {\Bbb B}^n/G$ be the
natural projection of ${\Bbb B}^n$ onto the factor space ${\Bbb
B}^n/G,$ then the {\it hyperbolic measure} of the set $A\subset
{\Bbb B}^n/G$ is defined by the relation $V(P\cap\pi^{\,-1}(A)),$
where $P$ is a normal fundamental polyhedron of the
form~(\ref{eq1}). It is easy to verify by direct calculations that
the hyperbolic measure does not change under any map~$g\in{\cal{G
M}}({\Bbb B}^n).$

\medskip
For given elements $p_1, p_2\in {\Bbb B}^n/G,$ we set
\begin{equation}\label{eq2}
\widetilde{h}(p_1, p_2):=\inf\limits_{g_1, g_2\in G}h(g_1(z_1),
g_2(z_2))\,,
\end{equation}
where $p_i=G_{z_i}=\{\xi\in {\Bbb B}^n:\,\exists\, g\in G:
\xi=g(z_i)\},$ $i=1,2.$ In the latter case, the set $G_{x_i}$ is
called the {\it orbit} of the point $x_i,$ and $p_1$ and $p_2$ are
called the {\it orbits} of the points $z_1$ and $z_2,$ respectively.
The length of the path $\gamma:[a, b]\rightarrow {\Bbb B}^n/G$ on
the segment $[a, t],$ $a\leqslant t\leqslant b,$ is defined as
follows:
\begin{equation}\label{eq5D}
l_{\gamma}(t):=\sup\limits_{\pi}\sum\limits_{k=0}^m\widetilde{h}(\gamma(t_k),
\gamma(t_{k+1}))\,,
\end{equation}
where~$\sup$ is taken over all partitions $\pi=\{a=t_0\leqslant
t_1\leqslant t_2\leqslant\ldots\leqslant t_m=b\}.$
Let $G$ be a group of M\"{o}bius automorphisms of the unit ball. We
say that $G$ acts {\it discontinuously} on ${\Bbb B}^n,$ if each
point $x\in {\Bbb B}^n$ has a neighborhood $U$ such that $g(U)\cap
U=\varnothing$ for all $g\in G,$ except maybe a finite number of
elements $g.$ We say that $G$ does not have {\it fixed points} in $
{\Bbb B}^n,$ if for any $a\in{\Bbb B}^n$ the equality $g(a)=a$ is
possible only when $g=I.$

\medskip
Let $D$ and $D_{\,*}$ be domains on the factor spaces ${\Bbb B}^n/G$
and ${\Bbb B}^n/G_{\,*},$ respectively. Suppose that ${\Bbb B}^n/G$
and ${\Bbb B}^n/G_{\,*}$ are metric spaces with metrics
$\widetilde{h}$ and $\widetilde{h_*},$ respectively. Hereinafter,
$\widetilde{h}$ and $\widetilde{h_*}$ are determined solely by
relation~(\ref{eq2}). (It will be established below, under what
conditions on $G$ and $G_*$ the indicated functions $\widetilde{h}$
and $\widetilde{h_*}$ are really metrics). Let $ds_{\widetilde{h}}$
and $d\widetilde{h}$ be elements of length and volume on ${\Bbb
B}^n/G,$ and let $ds_{\widetilde{h_*}}$ and $d\widetilde{h_*}$ be
elements of length and volume on ${\Bbb B}^n/G_*,$ correspondingly.

\medskip
The mapping $f:D\rightarrow D_{\,*}$ will be called {\it discrete},
if the preimage $f^{-1}\left(y\right)$ of each point $y\in D_{\,*}$
consists of isolated points only. The mapping $f:D\rightarrow
D_{\,*}$ will be called {\it open}, if the image of any open set
$U\subset D$ is an open set in $D_{\,*}.$

\medskip
All further presentation of the text of the article is based on the
following fundamental fact.

\medskip
\begin{proposition}\label{pr1A}
{\sl Suppose that $G$ is a group of M\"{o}bius transformations of
the unit ball ${\Bbb B}^n,$ $n\geqslant 2,$ onto itself, acting
discontinuously and not having fixed points in ${\Bbb B}^n.$ Then
the factor space ${\Bbb B}^n/G$ is a conformal manifold, that is, a
topological manifold in which any two charts are interconnected by
means of conformal mappings. At the same time, the natural
projection $\pi,$ which maps ${\Bbb B}^n$ onto ${\Bbb B}^n/G,$ is a
local homeomorphism. Moreover, the corresponding pairs of the form
$(U, \pi^{\,-1}),$ where $U$ is some neighborhood of an arbitrary
point $p\in {\Bbb B}^n/G,$ in which the mapping $\pi^{\,-1}$ is
well-defined and continuous, can be considered as charts
corresponding to the specified manifold. }
\end{proposition}

\medskip
\begin{proof}
By virtue of~\cite[Proposition~3.14]{Ap}, the space ${\Bbb B}^n/G$
is a topological manifold, and the map $\pi$ is a local
homeomorphism of the unit ball onto  ${\Bbb B}^n/G.$

\medskip
Suppose that $U_1$ and $U_2$ are open neighborhoods in ${\Bbb
B}^n/G$ such that $U_1\cap U_2\ne\varnothing$ and, besides that,
$\pi(W_i)=U_i$ for some open sets $W_1, W_2\subset {\Bbb B}^n,$
while $\pi|_{W_i}$ is a homeomorphism, $i=1,2.$ Denote
$\pi^{\,-1}_i:=(\pi|_{W_i})^{\,-1},$ $i=1,2.$ Let
$A_i=\pi_i^{\,-1}(U_1\cap U_2),$ $i=1,2.$ To complete the proof, we
need to establish that the mapping $\varphi:=\pi_1^{\,-1}\circ\pi_2$
maps $A_2$ onto $A_1$ conformally.

\medskip
Obviously, the mapping $\varphi$ is a homeomorphism. Fix an
arbitrary point $x_0\in A_2.$ Let $\pi(x_0)=y_0\in U_1\cap U_2$ and
let $\pi^{\,-1}_1(y_0)=z_0,$ then $z_0\in A_1$ and $\pi(z_0)=y_0.$
Since $\pi(x_0)=y_0 =\pi(z_0),$ then the points $x_0\in {\Bbb B}^n$
and $z_0\in{\Bbb B}^n$ belong to the same orbit $G_{x_0}.$ Thus, the
points $x_0$ and $z_0$ are interconnected by some M\"{o}bius
transformation $g \in G,$ that is, $g(x_0)=z_0.$

\medskip
Since simultaneously $g(x_0)=z_0$ and
$\varphi(x_0)=(\pi_1^{\,-1}\circ\pi_2)(x_0)=z_0,$ then
$g(x_0)=\varphi(x_0).$ Note that the mapping $g$ coincides with
$\varphi$ not only at one point $x_0,$ but also in some of its
neighborhoods. Indeed, since $g$ is homeomorphic, it maps some open
neighborhood $B_2$ of point $x_0$ onto some open neighborhood $B_1$
of point $z_0,$ and we may assume that $B_1\subset A_1$ and
$B_2\subset A_2.$ Let $x\in B_2,$ then $g(x)=z\in B_1.$ On the other
hand, note that
$$\pi_1(z)=\pi(z)=\pi(x)=\pi_2(x)$$
by definition of the mapping $\pi,$ therefore
$z=(\pi_1^{\,-1}\circ\pi_2)(x)=\varphi(x).$ We again have that
simultaneously $g(x)=\varphi(x),$ and this holds for an arbitrary
$x\in B_2.$ See Figure~\ref{fig2} for this.
\begin{figure}[h]
\centerline{\includegraphics[scale=0.65]{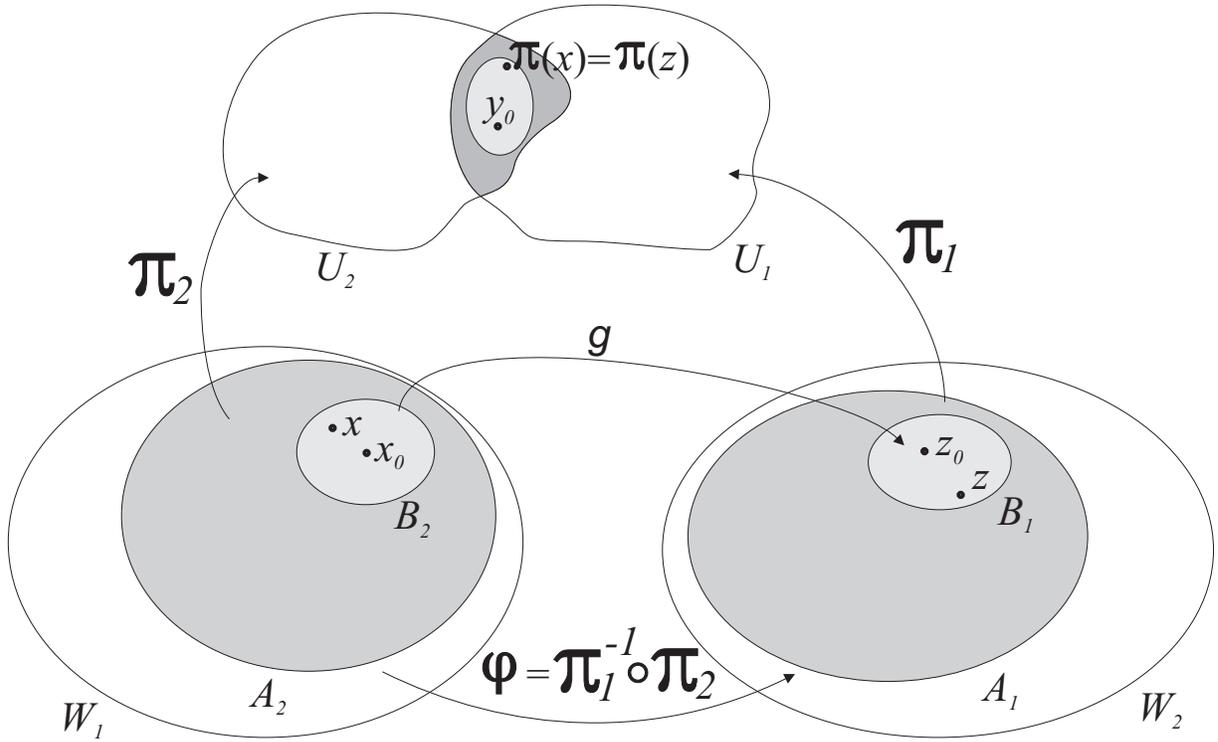}} \caption{To
the proof of Proposition~\ref{pr1A}}\label{fig2}
\end{figure}

\medskip
Finally, $\varphi$ is a homeomorphism of $A_2$ onto $A_1,$ which in
some neighborhood $B_2$ of an arbitrary point $x_0\in A_2$ coincides
with a conformal transformation. Therefore, $\varphi $ is a
conformal mapping, as required.~$\Box$
\end{proof}

\medskip
Suppose that $G$ is a group of M\"{o}bius transformations of the
unit ball ${\Bbb B}^n,$ $n\geqslant 2,$ onto itself, acting
discontinuously and not having fixed points in ${\Bbb B}^n.$

\medskip
Let $D$ and $D_{\,*}$ be domains on the factor spaces ${\Bbb B}^n/G$
and ${\Bbb B}^n/G_{\,*},$ respectively. We say that $f\in W_{\rm
loc}^{1,1}(D),$ if for each point $x_0\in D$ there are open
neighborhoods $U$ and $W,$ containing the points $x_0$ and $f(x_0),$
respectively, in which the natural projections
$\pi:(\pi^{-1}(U))\rightarrow U$ and
$\pi_*:(\pi_*^{-1}(V))\rightarrow V$ are one-to-one mappings, while
$F=\pi_*^{\,-1}\circ f\circ\pi\in W_{\rm loc}^{1,1}(\pi^{\,-1}(U)).$

\medskip
We write $f\in W_{\rm loc}^{1,p}(D),$ $p\geqslant 1,$ if $f\in
W_{\rm loc}^{1,1}(D)$ and, in addition, $\frac{\partial
f_i}{\partial x_j}\in L^p_{\rm loc}(D)$ in local coordinates. For a
given mapping $f:D\rightarrow{\Bbb R}^n,$ which is differentiable
almost everywhere in $D,$ we set
$$\Vert
f^{\,\prime}(x)\Vert=\max\limits_{|h|=1}{|f^{\,\prime}(x)h|}\,,\quad
l\left(f^{\,\prime}(x)\right)=\min\limits_{|h|=1}
{|f^{\,\prime}(x)h|}\,,\quad J(x, f)={\rm det}\,
f^{\,\prime}(x)\,.$$
The {\it inner dilatation} $K_I(x, f)$ of the mapping $f$ at the
point $x$ is defined by the relation
\begin{equation}\label{eq0.1.1}
K_I(x,f)\quad =\quad\left\{
\begin{array}{rr}
\frac{|J(x,f)|}{{l\left(f^{\,\prime}(x)\right)}^n}, & J(x,f)\ne 0,\\
1,  &  f^{\,\prime}(x)=0, \\
\infty, & \text{otherwise}
\end{array}
\right.\,.
\end{equation}
If we are talking about the mapping $f$ of domains $D$ and $D_*,$
belonging to the factor-spaces  ${\Bbb B}^n/G$ and ${\Bbb B}^n/G_*,$
respectively, then we set $K_I(p, f)=K_I(\varphi(p), F),$ where
$F=\psi\circ f\circ\varphi^{-\,1},$ $(U, \varphi)$ are local
coordinates of $x$ and $(V, \psi)$ are local coordinates of~$f(x).$
By virtue of Proposition~\ref{pr1A}, the mappings $\pi$ and $\pi_*$
can be considered as such local coordinates. By virtue of the same
statement, the definition of $K_I(p, f)$ does not depend on the
choice of local coordinates, because the inner dilatation of the
conformal mapping is equal to one.

\medskip
A {\it path} $\gamma$ on ${\Bbb B}^n/G$ is defined as a continuous
mapping $\gamma:I\rightarrow {\Bbb B}^n/G,$ where $I$ is a finite
segment, interval, or half-interval on the real axis. Let $\Gamma$
be a family of paths in ${\Bbb B}^n/G.$ A Borel function $\rho:{\Bbb
B}^n/G\rightarrow [0, \infty]$ is called an {\it admissible} for
$\Gamma,$ abbr. $\rho\in {\rm adm}\,\Gamma,$ if
$\int\limits_{\gamma}\rho(p)\,ds_{\widetilde{h}}(p)\geqslant 1$ for
each (locally rectifiable) path $\gamma\in\Gamma.$ The {\it modulus}
of $\Gamma$ is defined as follows:
$$M(\Gamma):=\inf\limits_{\rho\in {\rm adm}\,\Gamma}\int\limits_{{\Bbb
B}^n/G}\rho^n(p)\,d\widetilde{h}(p)\,.$$
Let $\Delta \subset \Bbb R$ be some open interval of the real axis,
and let $\gamma:  \Delta\rightarrow {\Bbb B}^n/G$ be a locally
rectifiable path. In this case, obviously, there is a unique
non-decreasing function of length $l_{\gamma}: \Delta\rightarrow
\Delta_{\gamma}\subset \Bbb{R}$ with the condition
$l_{\gamma}(t_0)=0,$ $t_0 \in \Delta,$ while $l_{\gamma}(t)$ is
equal to the length of  $\gamma\mid_{[t_0, t]}$ for $t>t_0,$ and to
$-\gamma\mid_ {[t,\,t_0]}$ for $t<t_0,$ $t\in \Delta.$ Let
$g:|\gamma|\rightarrow {\Bbb B}^n/G_*$ be a continuous mapping,
where $|\gamma| = \gamma(\Delta)\subset {\Bbb B}^n/G.$ Suppose that
the path $\widetilde{\gamma}=g\circ \gamma$ is also locally
rectifiable. Then, obviously, there is a unique non-decreasing
function $L_{\gamma,\,g}: \,\Delta_{\gamma} \rightarrow
\Delta_{\widetilde{\gamma}}$ such that
$L_{\gamma,\,g}(l_{\gamma}(t))\,=\,l_{\widetilde{\gamma}}(t)$  for
all $t\in\Delta.$ If $\gamma$ is defined on the segment $[a, b]$ or
the half-interval $[a, b),$ then we assume that $a=t_0.$ A path
$\gamma $ is called a {\it (total) lifting} of $\widetilde{\gamma}$
under the mapping $f:D\rightarrow {\Bbb B}^n/G_*,$ if
$\widetilde{\gamma}=f \circ \gamma.$

\medskip
Throughout, we say that a certain property $P$ holds for {\it almost
all paths,} if the modulus of the family of paths for which this
property can be violated is zero. The following definition can be
found in~\cite[определение~5.2]{Va$_3$}
or~\cite[разд.~8.4]{MRSY$_2$}. We say that the mapping $f:
D\rightarrow {\Bbb R}^n$ is {\it absolutely continuous on almost all
paths in $D$}, abbr. $f\in ACP, $ if, for almost all $\gamma$ in
$D,$ $\widetilde{\gamma}=f\circ\gamma$ is locally rectifiable and
the function $L_{\gamma,\,f}$ is absolutely continuous on all closed
segments lying in $\Delta_{\gamma}.$

\medskip
Let $f:D\rightarrow {\Bbb B}^n/G_*$ be a mapping such that no path
$\alpha$ in $D$ transforms into a point under $f.$ Note that in this
case, the inverse function is $L^{\,-1}_{\gamma,\,f}$ is
well-defined. Now, we say that $f$ is {\it absolutely continuous on
the paths in the inverse direction}, abbr. $f\in ACP^{\,-1},$ if
each lift $\gamma$ of $\widetilde{\gamma}=f\circ\gamma$ is locally
rectifiable and, in addition, $L^{-1}_{\gamma,\,f}$ is absolutely
continuous on all closed segments lying in
$\Delta_{\widetilde{\gamma}}$ for almost all paths
$\widetilde{\gamma}$ in $f(D).$

\medskip
The properties of the path to be absolutely continuous under the
mapping, as well as the property of the path to be absolutely
continuous in the preimage under $f,$ is crucial for establishing
estimates of the distortion of the modulus under $f$ see, for
example,~\cite[Lemma~5.1]{Ri} or~\cite[Lemma~6]{Pol}. If we are
talking about estimates of the distortion of the modulus from above,
then to establish them, we use, as a rule, absolute continuity of
the paths in the ''inverse'' direction, at the same time, for the
lower distortion estimates, we need to use $ACP$-property;  see, for
example, \cite[Theorems~8.5-8.6]{MRSY$_2$}. Note that if $f$ is a
homeomorphism such that $f^{\,-1}\in W_{\rm loc}^{1, n}(f(D)),$ then
$f\in ACP^{\,-1},$ see~\cite[Theorem~28.2]{Va$_3$}. We say that a
mapping $f$ has {\it $N$-Luzin property,} if
$\widetilde{h_*}(f(E))=0$ for any $E \subset D$ such that
$\widetilde{h}(E)=0.$ Similarly, we say that the mapping $f$ has
{\it $N^{\,-1}$-Luzin property,} if $\widetilde{h}(f^{\,-1}(E_*))=0$
for any $E_*\subset D_*$ such that $\widetilde{h_*}(E_*)=0.$ The
concept of a topological index on a smooth manifold, used below, can
be found, for example, in~\cite{Va$_2$}. We write $\alpha \subset
\beta$ for paths $\alpha: I\rightarrow {\Bbb B}^n/G$  and
$\beta:J\rightarrow {\Bbb B}^n/G$  if $I$ is a certain sub-interval
of the interval and $J.$ One of the main assertions of this paper is
the following theorem (see also~\cite[Lemma~3.1]{RV$_1$}).

\medskip
\begin{theorem}\label{th1}{\sl\, Suppose that $G$ and $G_*$
are two groups of M\"{o}bius transformations of the unit ball ${\Bbb
B}^n,$ $n\geqslant 2,$ onto itself, acting discontinuously on ${\Bbb
B}^n$ and not having fixed points in ${\Bbb B}^n.$ Let $D$ and
$D_{\, *}$ be domains belonging to ${\Bbb B}^n/G$ and ${\Bbb
B}^n/G_*,$ respectively, and having compact closures $\overline{D}$
and $\overline{D_*}.$ Let $I$ be open, half-open or closed finite
interval of the real axis, and let $f:D\rightarrow D_{\,*}$ be an
open discrete almost everywhere differentiable map, $f\in ACP^{\, -
1},$ having $N$ and $N^{\,-1}$-Luzin properties.

Suppose that $\Gamma $ is a family of paths in $D,$
$\Gamma^{\,\prime}$ is a family of paths in ${\Bbb B}^n/G_*$ and
$\widetilde{m}$ is some positive integer number such that the
following condition is satisfied. For each $\beta\in
\Gamma^{\,\prime},$ $\beta:I\rightarrow D_{\,*},$ there are paths
$\alpha_1,\ldots,\alpha_{\widetilde{m}}$ in $\Gamma$ such that
$f\circ \alpha_j\subset \beta$ for all $j=1,\ldots, \widetilde{m}$
and, moreover, for each fixed $p\in D$ and $t\in I,$ the equality
$\alpha_j(t)=p$ is possible no more than with $i(p, f)$ indices $j.$
Then
\begin{equation}\label{eq1A}
M(\Gamma^{\,\prime})\leqslant
\frac{1}{\widetilde{m}}\int\limits_DK_I(p,
f)\cdot\rho^n(p)\,d\widetilde{h}(p)\qquad \forall\,\,\rho\in{\rm
adm\,}\Gamma\,.
\end{equation}
}
\end{theorem}

\medskip
By virtue of~\cite[Theorem~28.2]{Va$_3$},  \cite[Corollary~B]{MM}
and~\cite[Lemma~3]{Va$_1$}, we also have the following

\medskip
\begin{corollary}\label{cor1}
{\sl\, Suppose that $G$ and $G_*$ are two groups of M\"{o}bius
transformations of the unit ball ${\Bbb B}^n,$ $n\geqslant 2,$ onto
itself, acting discontinuously on ${\Bbb B}^n$ and not having fixed
points in ${\Bbb B}^n.$ Let $D$ and $D_{\, *}$ be domains belonging
to ${\Bbb B}^n/G$ and ${\Bbb B}^n/G_*,$ respectively, and having
compact closures $\overline{D}$ and $\overline{D_*}.$ Let $f$ be a
mapping of $D$ onto $D_*$ such that $f\in W_{\rm loc}^{1, n}(D)$ and
$f^{\,-1}\in W_{\rm loc}^{1, n}(f(D)).$ Then the
relation~(\ref{eq1A}) holds. }
\end{corollary}

\medskip
Obviously, for $\widetilde{m}=1,$ Theorem~\ref{th1} also implies the
following

\medskip
\begin{corollary}\label{cor3}
{\sl\, Suppose that $G$ and $G_*$ are two groups of M\"{o}bius
transformations of the unit ball ${\Bbb B}^n,$ $n\geqslant 2,$ onto
itself, acting discontinuously on ${\Bbb B}^n$ and not having fixed
points in ${\Bbb B}^n.$ Let $D$ and $D_{\, *}$ be domains belonging
to ${\Bbb B}^n/G$ and ${\Bbb B}^n/G_*,$ respectively, and having
compact closures $\overline{D}$ and $\overline{D_*}.$ Let
$f:D\rightarrow D_{\,*}$ be an open discrete almost everywhere
differentiable map, $f\in ACP^{\,-1},$ having $N$ and
$N^{\,-1}$-Luzin properties. Suppose that $\Gamma$ is a family of
paths in $D.$ Then
\begin{equation}\label{eq1C}
M(f(\Gamma))\leqslant \int\limits_DK_I(p,
f)\cdot\rho^n(p)\,d\widetilde{h}(p)\qquad \forall\,\,\rho\in{\rm
adm\,}\Gamma\,.
\end{equation}
}
\end{corollary}

\section{Preliminaries}
Before turning to auxiliary assertions and proof of the main
results, we make some important remarks. For a point $y_0\in {\Bbb
B}^n$ and a number $r\geqslant 0$ we define the {\it hyperbolic
ball} $B_h(y_0, r)$ and the {\it hyperbolic sphere} $S_h(y_0, r)$ by
equalities
\begin{equation}\label{eq7}
B_h(y_0, r):=\{y\in {\Bbb B}^n: h(y_0, y)<r\}\,, S_h(y_0, r):=\{y\in
{\Bbb B}^n: h(y_0, y)=r\}\,. \end{equation}
We start with the following simple lemma that the exact lower bound
of a quantity in equality~(\ref{eq2}) is always achieved on a
certain pair of elements.

\medskip
\begin{lemma}\label{lem5}
{\sl Suppose that $G$ is a group of M\"{o}bius transformations of
the unit ball ${\Bbb B}^n,$ $n\geqslant 2,$ onto itself, acting
discontinuously and not having fixed points in ${\Bbb B}^n.$ Then
$\inf$ in~(\ref{eq2}) is attained on some pair of maps $f, g\in G,$
in other words, for each pair of points $p_1=G_{z_1}, p_2=G_{z_2}\in
{\Bbb B}^n/G$ and arbitrary~$z_1\in G_{z_1}, z_2\in G_{z_2}$ there
are mappings $f, g\in G$ such that
\begin{equation}\label{eq5B}\widetilde{h}(p_1, p_2)=\min\limits_{g_1, g_2\in G}h(g_1(z_1),
g_2(z_2))=h(f(z_1), g(z_2))\,.
\end{equation}
}
\end{lemma}

\begin{proof}
Fix arbitrarily $p_1=G_{z_1}, p_2=G_{z_2}\in {\Bbb B}^n/G$ and
$z_1\in G_{z_1}, z_2\in G_{z_2}.$ By virtue of invariance of a
hyperbolic metric $h$ under M\"{o}bius transformations
\begin{equation}\label{eq6A}
M:=\widetilde{h}(p_1, p_2)=\inf\limits_{g\in G}h(g(z_1), z_2)\,.
\end{equation}
From~(\ref{eq6A}) it follows that there is a sequence $g_k\in G$
such that $h(g_k(z_1), z_2)\rightarrow M$ as $k\rightarrow\infty,$
where $M <\infty,$ because $M\leqslant h(z_1, z_2).$ Since
$\overline{{\Bbb B}^n}$ is a compactum in~${\Bbb R}^n,$ we may
assume that the sequence $g_k(z_1)$ converges to a certain point
$z_0\in \overline{{\Bbb B}^n},$ in this case, $z_0\in {\Bbb B}^n,$
because otherwise $h(g_k(z_1), z_2)\rightarrow\infty$ by the
definition of the function $h$ in~(\ref{eq3}). Choose an arbitrary
neighborhood $U$ of a point~$z_0,$ then $g_k(z_1)\in U$ for all
$k\geqslant k_0 $ and some $k_0=k_0(U)\in {\Bbb N}.$ Note that for
the indicated $k$ the elements $g_k(z_1)$ and $g_{k+1}(z_1)$ belong
to $U,$ and $g_{k+1}(z_1)=f_k(g_k(z_1))$ for some $f_k\in G,$ since
$g_k$ and $g_{k+1}\in G.$ Then~$g_{k+1}(z_1)\in U\cap f_k(U),$ which
contradicts the discontinuity of the group $G$ together with the
condition of the absence of fixed points, if only $f_k(z)\ne I.$
Then $g_k(z)=g(z)$ for $k\geqslant k_0$ and $g\in G.$
From~(\ref{eq6A}) it follows that~$M=h(g(z_1), z_2),$ $g\in G.$ The
lemma is proved.~$\Box$
\end{proof}

\medskip
In what follows
$$\widetilde{B}(p_0, r):=\{p\in {\Bbb B}^n/G: \widetilde{h}(p_0, p)<r\}\,,\quad \widetilde{S}(p_0, r):=\{p\in
{\Bbb B}^n/G: \widetilde{h}(p_0, p)=r\}\,.$$

\medskip
The most important element of further research is the metrizability
of the space ${\Bbb B}^n/G.$ Of course, in the most general case,
factor-spaces are not metrizable by an appropriate way, besides,
they are also not, generally speaking, manifolds
see~\cite[section~3.3]{MS$_1$}. If we are talking about a discrete
group $G,$ then such circumstances may be associated with the
presence of elliptic and parabolic elements in it, that is, with a
situation when there are elements with fixed points in the group. As
for the groups of maps acting discontinuously and not having fixed
points, we have the following statement.

\medskip
\begin{lemma}\label{lem4}
{\sl Suppose that $G$ is a group of M\"{o}bius transformations of
the unit ball ${\Bbb B}^n,$ $n\geqslant 2,$ onto itself, acting
discontinuously and not having fixed points in ${\Bbb B}^n.$ Then
the space~${\Bbb B}^n/G$ is metrizable, and the corresponding metric
can be determined by relation~(\ref{eq2}).}
\end{lemma}

\medskip
\begin{proof}
We show that $\widetilde{h}$ in~(\ref{eq2}) is a metric on~${\Bbb
B}^n/G.$ First of all,  $\widetilde{h}(p_1, p_2)\geqslant 0$ for all
$p_1, p_2\in {\Bbb B}^n/G.$ Suppose now that $\widetilde{h}(p_1,
p_2)=0.$ Then, by Lemma~\ref{lem5} there are $z_1,z_2\in {\Bbb B}^n$
and $f, g\in G$ such that
$$\widetilde{h}(p_1,
p_2)=h(f(z_1), g(z_2))=0\,.$$ Then $z_1=(f^{\,-1}\circ g)(z_2),$
i.e., the points $z_1$ and $z_1$ belong to the same orbit, whence
$p_1=p_2.$

The symmetry of the function $\widetilde{h}$ in~(\ref{eq2}) is
obvious. We show the triangle inequality.  Let $p_1, p_2, p_3\in
{\Bbb B}^n/G.$ Then, by Lemma~\ref{lem5}, there are $g_1, g_2\in G,$
such that
$\widetilde{h}(p_1, p_2)=h(g_1(z_1), z_2)$ and $\widetilde{h}(p_2,
p_3)=h(z_2, g_2(z_3)),$ where points $z_1,$ $z_2$ and $z_3$ belong
to the orbits $p_1,$ $p_2$ and $p_3,$ respectively. Then, by the
triangle inequality for the metric $h,$ we have:
$$\widetilde{h}(p_1, p_3)=\inf\limits_{f_1, f_2\in G}h(f_1(z_1), f_2(z_3))\leqslant$$
$$\leqslant h(g_1(z_1), g_2(z_3))\leqslant h(g_1(z_1), z_2)+h(z_2, g_2(z_3))=$$
$$=\widetilde{h}(p_1, p_2)+\widetilde{h}(p_2, p_3)\,.$$

\medskip
It remains to show that the metric $\widetilde{h}$ generates the
same topology on~${\Bbb B}^n/G$ as the original topological
space~${\Bbb B}^n/G.$ By Proposition~\ref{pr1A}, it suffices to show
that the metric space~$({\Bbb B}^n/G, \widetilde{h})$ is locally
homeomorphic to~${\Bbb R}^n.$ By Proposition~\ref{pr1A}, every point
$p_0\in {\Bbb B}^n/G$ has a neighborhood $U$ and a homeomorphism
$\varphi$ such that $\varphi(U)=W$ and $W$ is open in~${\Bbb R}^n.$
By the same proposition, we may take~$\varphi:=(\pi^{\,-1})|_{U},$
$\pi^{\,-1}:U\rightarrow W.$ Let $\varphi(p_0)=x_0\in W.$ Choose
$r_0>0$ so that $B_h(x_0, r_0)\subset U.$ Then
$$\varphi(\widetilde{B}(p_0, r))=B_h(x_0, r)\,,$$
whence it follows that the ball~$\varphi(\widetilde{B}(p_0, r))$ is
an open set in ${\Bbb B}^n/G$ in the sense of topology~${\Bbb
B}^n/G.$ The lemma is proved. ~$\Box$
\end{proof}

Given $z_1, z_2\in {\Bbb B}^n,$ we set
\begin{equation}\label{eq5}
d(z_1, z_2):=\widetilde{h}(\pi(z_1), \pi(z_2))\,,
\end{equation}
where $\widetilde{h}$ is defined in~(\ref{eq2}). Note that, by the
definition, $d(z_1, z_2)\leqslant h(z_1, z_2).$ The following
statement is true.

\medskip
\begin{lemma}\label{lem6}
{\sl Suppose that $G$ is a group of M\"{o}bius transformations of
the unit ball ${\Bbb B}^n,$ $n\geqslant 2,$ onto itself, acting
discontinuously and not having fixed points in ${\Bbb B}^n.$ Then
for any compact set $A\subset {\Bbb B}^n$ there is
$\delta=\delta(A)>0$ such that
\begin{equation}\label{eq34}
d(z_1, z_2)=h(z_1, z_2), \quad \forall\,\, z_1, z_2\in A: h(z_1,
z_2)<\delta\,.
\end{equation}
}
\end{lemma}
\begin{proof}
Suppose the contrary. Then for an arbitrary $k\in {\Bbb N}$ there
are $x_k, z_k\in A$ such that $h(z_k, x_k)<1/k$ and, moreover,
$d(z_k, x_k)<h(z_k, x_k).$ Thus, by the definition of the metric $d$
and the invariance of the metric $h$ under M\"{o}bius
transformations of the unit ball onto itself, there exists $g_k\in
G,$ $g_k\ne I,$ such that
\begin{equation}\label{eq37}
d(z_k, x_k)\leqslant h(z_k, g_k(x_k))<h(z_k, x_k)<1/k,\quad g_k\in
G, \quad k=1,2,\ldots \,.
\end{equation}
Since $A$ is a compact in~${\Bbb B}^n,$ we may assume that $x_k,
z_k\rightarrow x_0\in {\Bbb B}^n$ as $k\rightarrow\infty.$ In this
case, from~(\ref{eq37}) by the triangle inequality we obtain that
$h(g_k(x_k), x_0)\leqslant h(g_k(x_k), z_k)+h(z_k, x_k)\rightarrow
0$ as $k\rightarrow\infty,$ and, therefore, $h(x_k,
g_k^{\,-1}(x_0))\rightarrow 0$ as $k\rightarrow\infty,$ since the
metric $h$ is invariant under M\"{o}bius transformations. But then
also, by the triangle inequality, $h(g_k^{\,-1}(x_0), x_0)\leqslant
h(g_k^{\,-1}(x_0), x_k)+h(x_k, x_0)\rightarrow 0,$
$k\rightarrow\infty.$ The latter contradicts the discontinuity of
the group $G$ in ${\Bbb B}^n$ together with the absence of fixed
points, which proves~(\ref{eq34}). ~$\Box$
\end{proof}

\medskip
In many cases, it is important to have estimates of the hyperbolic
distance through the Euclidean distance. In some cases, such
estimates are more or less obvious, in some they require significant
maintenance. To achieve clarity in this matter, we establish the
following assertion.

\medskip
\begin{lemma}\label{lem3}
{\sl Let $0<2r_0<1,$ then there is $C_1=C_1(r_0)>0$ such that
\begin{equation}\label{eq9B}
C_1\cdot h(z_1, z_2)\leqslant |z_1-z_2|\leqslant  h(z_1,
z_2)\quad\forall\,\, z_1, z_2\in B(0, r_0)\,. \end{equation}
Moreover, the right-hand inequality in~(\ref{eq9B}) holds for all
$z_1, z_2\in {\Bbb B}^n.$ }
\end{lemma}

\medskip
\begin{proof}
Note that by the triangle inequality $0<|z_1-z_2|<r_0.$ Therefore,
$r:=|z_1-z_2|$ varies from $0$ to $2r_0<1.$ Recall that
$$h(z_1,
z_2)=\log\frac{1+\frac{|z_1-z_2|}{\sqrt{|z_1-z_2|^2+(1-|z_1|^2)(1-|z_2|^2)}}}{1-
\frac{|z_1-z_2|}{\sqrt{|z_1-z_2|^2+(1-|z_1|^2)(1-|z_2|^2)}}}\,.$$
Through direct computation, we make sure that
$|z_1-z_2|^2+(1-|z_1|^2)(1-|z_2|^2)\leqslant 4$ for all $z_1, z_2\in
{\Bbb B}^n$ and, consequently, $h(z_1, z_2)\geqslant
\log\frac{1+r/2}{1-r/2},$ where $r=|z_1-z_2|.$ Note that
\begin{equation}\label{eq8C}
h(z_1, z_2)\geqslant \log\frac{1+r/2}{1-r/2}\geqslant r\,,\quad r\in
(0, 1)\,.
\end{equation}
By taking a derivative, it can be shown that the function
$\varphi(r)=\log\frac{1+r/2}{1-r/2}-r$ increases by $r\in [0, 1].$
Hence, its minimum is reached at $r=0,$ that is,
$\varphi(r)\geqslant 0$ for all $r\in (0, 1).$ Thus,
inequality~(\ref{eq8C}) is established.

Now we establish the left inequality in~(\ref{eq9B}). Considering
Cauchy-Bunyakovsky inequality, note that
$$|z_1-z_2|^2+(1-|z_1|^2)(1-|z_2|^2)=|z_1|^2|z_2|^2-2(z_1,
z_2)+1\geqslant$$
$$
\geqslant 1-|z_1|^2|z_2|^2\geqslant 1-r_0^4\geqslant 1-r_0^2$$
for $z_1, z_2 \in B(0, r_0).$ It follows that
$$\sqrt{|z_1-z_2|^2+(1-|z_1|^2)(1-|z_2|^2)}\geqslant 1-r_0^2\,.$$
Hence $h(z_1, z_2)\leqslant \log\frac{1-r^2_0+r}{1-r^2_0-r}$
for~$z_1, z_2\in B(0, r_0),$ moreover,
$\log\frac{1-r^2_0+r}{1-r^2_0-r}\sim \frac{2}{1-r^2_0}\cdot r$ as
$r\rightarrow 0.$ Then
$$h(z_1, z_2)\leqslant
\log\frac{1-r^2_0+r}{1-r^2_0-r}\leqslant M\cdot r\,,\quad r\in (0,
r_1)$$
for some $0<r_1<r_0$ and $M=M(r_0).$ The function $1-r^2_0-r$ is
strictly positive for $r\in [0, 1].$ Therefore, the function
$\frac1r\cdot\log\frac{1-r^2_0+r}{1-r^2_0-r}$ is continuous by $r\in
[r_1, 1]$ and, therefore, is bounded for the same $r$ with some
constant $\widetilde{C}.$ Putting $C^{\,-1}_1:=\max\{M,
\widetilde{C}\},$ we obtain that
\begin{equation}\label{eq8A}
h(z_1, z_2)\leqslant \log\frac{1-r^2_0+r}{1-r^2_0+r}\leqslant
C^{\,-1}_1\cdot r=C^{\,-1}_1\cdot |z_1-z_2|\qquad\forall\,\, z_1,
z_2\in B(0, r_0)\,.
\end{equation}
Lemma is proved.~$\Box$
\end{proof}

\medskip
In what follows, we set
\begin{equation}\label{eq5C}
s_h(\gamma):=\sup\limits_{\pi}\sum\limits_{k=0}^mh(\gamma(t_k),
\gamma(t_{k+1}))\,,
\end{equation}
where $\sup$ is taken over all partitions $\pi=\{a=t_0\leqslant
t_1\leqslant t_2\leqslant\ldots\leqslant t_m=b\}.$
Let us prove the following important statement,
generalizing~\cite[теорема~1.3(5)]{Va$_3$}.

\medskip
\begin{lemma}\label{lem2}
{\sl\, Let $\alpha:[a, b]\rightarrow {\Bbb B}^n$ be a path that is
locally rectifiable with respect to length $s_h$ in~(\ref{eq1}), and
let $s_h=s_h(t)$ be the hyperbolic length of $\alpha$ on $[a, t],$
where $a\leqslant t\leqslant b.$ Then $\alpha^{\,\prime}(t)$ and
$s_h^{\,\prime}(t)$ exist for almost all $t\in [a, b],$ wherein
\begin{equation}\label{eq8B}
\frac{2|\alpha^{\,\prime}(t)|}{1-|\alpha(t)|^2}=s_h^{\,\prime}(t)
\end{equation}
for almost all $t\in [a, b].$
 }
\end{lemma}

\medskip
\begin{proof}
The function $s_h=s_h(t)$ is monotone and, therefore, is almost
everywhere differentiable. In addition, since $\alpha$ is
rectifiable, there exists $0<r_0<1$ such that $\alpha(t)\in B(0,
r_0)$ for all $t\in [a, b].$ In this case, by Lemma~\ref{lem3}
$\alpha$ is also rectifiable in the Euclidean sense. Thus, $\alpha$
has bounded variation and, therefore, is also differentiable almost
everywhere. To establish the equality~(\ref{eq8B}), we will follow
the logic of the reasoning used in the proof
of~\cite[Theorem~1.3(5)]{Va$_3$}. First of all, based on definition
of the hyperbolic length in~~(\ref{eq5C}), we can write that
\begin{equation}\label{eq9}
\frac{h(\alpha(t), \alpha(t_0))}{|t-t_0|}\leqslant
\frac{|s_h(t)-s_h(t_0)|}{|t-t_0|}\,.
\end{equation}
Multiplying the numerator and denominator of the ratio~(\ref{eq9})
by $|\alpha(t)-\alpha(t_0)|,$ we obtain that
\begin{equation}\label{eq9A}\frac{|\alpha(t)-\alpha(t_0)|}{|\alpha(t)-\alpha(t_0)|}
\cdot \frac{h(\alpha(t), \alpha(t_0))}{|t-t_0|}\leqslant
\frac{|s_h(t)-s_h(t_0)|}{|t-t_0|}\,.
\end{equation}
Find out the behavior of the function $\varphi(t)=\frac{h(\alpha(t),
\alpha(t_0))}{|\alpha(t)-\alpha(t_0)|}$ as $t\rightarrow t_0.$ Since
$\log\frac{1+x}{1-x}\sim 2x$ as $x\rightarrow 0,$ then
$$\varphi(t)=\log\left(\frac{1+\frac{|\alpha(t)-\alpha(t_0)|}{{\sqrt{|\alpha(t)-\alpha(t_0)|^2+(1-|\alpha(t)|^2)(1-|\alpha(t_0)|^2)}}}}
{1-\frac{|\alpha(t)-\alpha(t_0)|}{{\sqrt{|\alpha(t)-\alpha(t_0)|^2+(1-|\alpha(t)|^2)(1-|\alpha(t_0)|^2)}}
}}\right)\cdot\frac{1}{|\alpha(t)-\alpha(t_0)|}\sim
$$
$$\sim \frac{2}{{\sqrt{|\alpha(t)-\alpha(t_0)|^2+(1-|\alpha(t)|^2)(1-|\alpha(t_0)|^2)}}}$$
as $t\rightarrow t_0.$ Then $\varphi(t)\rightarrow
\frac{2}{1-|\alpha(t_0)|^2}$ as $t\rightarrow t_0.$ In this case,
letting in~(\ref{eq9A}) to the limit as $t\rightarrow t_0,$ we
obtain that
\begin{equation}\label{eq10}
\frac{2|\alpha^{\,\prime}(t)|}{1-|\alpha(t_0)|^2}\leqslant
s_h^{\,\prime}(t)
\end{equation}
for almost all $t\in [a, b].$

To complete the proof, it remains to establish the opposite
inequality to~(\ref{eq10}). Denote by $A$ the set of all points of
$[a, b],$ for which $\alpha^{\,\prime}(t)$ и $s_h^{\,\prime}(t)$
exist and, wherein,
$$\frac{2|\alpha^{\,\prime}(t)|}{1-|\alpha(t_0)|^2}<
s_h^{\,\prime}(t)\,.$$ Let $A_k$ be the set of all points $t\in A,$
for which
$$\frac{s_h(q)-s_h(p)}{q-p}\geqslant \frac{h(\alpha(q),\alpha(p))}{q-p}+1/k\,,$$
where $a\leqslant p\leqslant t\leqslant q \leqslant b$ and
$0<q-p<1/k.$ Clearly, to complete the proof, it suffices to
establish that $m_1(A_k)=0$ for any $k=1,2,\ldots,$ where $m_1$ is
Lebesgue measure in ${\Bbb R}^1.$

Let $l(\alpha)$ denote the hyperbolic length of $\alpha.$ Given
$\varepsilon>0,$ consider the partitioning of the segment $[a, b]$
by points $a=t_1\leqslant t_2\leqslant\ldots\leqslant t_m=b,$ so
that $l(\alpha)\leqslant \sum\limits_{k=1}^m h(\alpha(t_k),
\alpha(t_{k-1}))+\varepsilon/k$ и $t_{j}-t_{j-1}<1/k$ for every
$j=1,2,\ldots, m.$ If $[t_{j-1}, t_j]\cap A_k\ne \varnothing,$ then
by definition of the set $A_k,$ $s_h(t_j)-s_h(t_{j-1})\geqslant
h(\alpha(t_j), \alpha(t_{j-1}))+(t_{j}-t_{j-1})/k.$ Therefore,
denoting $\Delta_j:=[t_{j-1}, t_j],$ we obtain that
$$m_1(A_k)\leqslant\sum\limits_{\Delta_j\cap A_k\ne \varnothing}m_1(\Delta_j)\leqslant
k\sum\limits_{j=1}^m (s_h(t_j)-s_h(t_{j-1})- h(\alpha(t_j),
\alpha(t_{j-1}))) \leqslant$$
$$\leqslant k\left(l(\alpha)-
\sum\limits_{j=1}^m h(\alpha(t_j), \alpha(t_{j-1}))\right)\leqslant
\varepsilon\,.$$
The last relation proves the equality $m_1(A_k)=0$ and, since
$A=\bigcup\limits_{k=1}^{\infty} A_k,$ then $m_1(A)=0,$ as required.
~$\Box$
\end{proof}

\medskip
Let $ I $ be an open, closed, or semi-closed finite interval.
According to~\cite[Section~7.1]{He} and~\cite[Theorem~2.4]{Va$_3$},
each rectifiable path $\gamma:I\rightarrow {\Bbb R}^n$
(respectively, $\gamma:I\rightarrow {\Bbb B}^n/G$) admits a
parametrization $\gamma(t)=(\gamma^0\circ l_\gamma)(t),$ where
$l_\gamma$ denotes the length of the path $\gamma$ on $[a, t].$
Depending on the context, this length can be understood both in the
Euclidean and hyperbolic sense, as well as in the sense of the
factor space ${\Bbb B}^n/G.$ In this case, the path $\gamma^0:[0,
l(\gamma)]\rightarrow {\Bbb R}^n$ (respectively, $\gamma^0:[0,
l(\gamma)]\rightarrow {\Bbb B}^n/G$) is unique and is called {\it
the normal representation} of~$\gamma.$ Let $\gamma:[a,
b]\rightarrow {\Bbb B}^n$ be a locally rectifiable path. Then we set
$$\int\limits_{\alpha}\rho(x)\,ds_h(x)=\int\limits_0^{l(\gamma)}\rho(\alpha^0(s))\,ds\,.$$
Based on the above, by Lemma~\ref{lem2} we obtain the following
assertion.

\medskip
\begin{corollary}\label{cor2}
{\sl\,Let $\alpha:[a, b]\rightarrow {\Bbb B}^n$ be an absolutely
continuous path and let $\rho:{\Bbb B}^n\rightarrow {\Bbb R}$ be a
nonnegative Borel function. Then
\begin{equation}\label{eq14}
\int\limits_{\alpha}\rho(x)\,ds_h(x)=\int\limits_a^b\frac{2\rho(\alpha(t))
|\alpha^{\,\prime}(t)|}{1-|\alpha(t)|^2}\,dt\,, \end{equation}
in particular,
}
\end{corollary}

\medskip
Let $p_0\in {\Bbb B}^n/G$ and $z_0\in {\Bbb B}^n$ be such that
$\pi(z_0)=p_0,$ where $\pi$ is the natural projection of ${\Bbb
B}^n$ onto ${\Bbb B}^n/G.$ By Proposition~\ref{pr1A},
Lemma~\ref{lem4} and Lemma~\ref{lem6} we may choose
$\varepsilon_0>0$ such that $\pi$ maps $B_h(z_0,\varepsilon_0)$
onto~$\widetilde{B}(p_0, \varepsilon_0)$ homeomorphically. Without
loss of generality, we may assume that $z_0=0.$ Indeed, otherwise we
consider the auxiliary map
$T_{z_0}=p_{z_0}\circ\sigma_{z_0},$ where $\sigma_{z_0}$ be an
inversion in the sphere $S(z^*_0, r)$ orthogonal to ${\Bbb S}^{n-1}$
and $p_{z_0}$ denote the reflection in the $(n—1)$-dimensional plane
$P(z_0, 0)$ through the origin and orthogonal to $z_0.$ In other
words,
$\sigma_{z_0}(z)=z_0^*+r^2(x-z_0^*)^*,$ $r^2=|z_0^*|^{\,-2}-1,$
$z_0^*=\frac{z_0}{|z_0|^2},$ $p_{z_0}(x)=x-2(z_0,
x)\frac{z_0}{|z_0|^2}.$ It can be shown that the mapping $T_{z_0}$
transforms the unit ball onto itself, $T_{z_0}(z_0)=0,$ and,
moreover, $T$ has no fixed points in ${\Bbb B}^n,$
see~\cite[Section~1.34, Lemma~1.37]{Vu}. In the above notation, a
neighborhood $U$ of a point $p_0,$ lying together with its closure
in $\widetilde{B}(p_0, \varepsilon_0)$, we will call {\it a normal
neighborhood of a point} $p_0.$

\medskip
\begin{remark}\label{rem1}
Note that~~${\Bbb B}^n/G$ admits a countable covering consisting of
normal neighborhoods~$U_k.$ Indeed, ${\Bbb
B}^n/G=\bigcup\limits_{x\in{\Bbb B}^n/G}\widetilde{B}(x,
\varepsilon(x)),$ where $\varepsilon(x)>0$ is such that
$\widetilde{B}(x, \varepsilon(x))$ is a normal neighborhood of the
point~$x.$ Note that
$$\pi^{\,-1}({\Bbb
B}^n/G)=\bigcup\limits_{x\in{\Bbb B}^n/G}\pi^{\,-1}(\widetilde{B}(x,
\varepsilon(x)))={\Bbb B}^n\,,$$
and~$\pi^{\,-1}(\widetilde{B}(x, \varepsilon(x)))$ is open in~${\Bbb
B}^n$ by the continuity of $\pi.$ By the Lindel\"{o}f theorem,
see~\cite[Theorem~1.5.XI]{Ku}, we may choose the sequence $x_k\in
{\Bbb B}^n/G,$ $k=1,2,\ldots$ such that
$${\Bbb
B}^n=\bigcup\limits_{k=1}^{\infty}\pi^{\,-1}(\widetilde{B}(x_k,
\varepsilon(x_k)))\,.$$
Since $\pi(\pi^{\,-1}(\widetilde{B}(x_k,
\varepsilon(x_k))))=\widetilde{B}(x_k, \varepsilon(x_k))$ and, in
addition, $\pi({\Bbb B}^n)={\Bbb B}^n/G,$ it follows that
$$\pi({\Bbb
B}^n)=\bigcup\limits_{k=1}^{\infty}\pi(\pi^{\,-1}(\widetilde{B}(x_k,
\varepsilon(x_k))))=\bigcup\limits_{k=1}^{\infty}\widetilde{B}(x_k,
\varepsilon(x_k))={\Bbb B}^n/G\,,$$
as required.
\end{remark}

\medskip
Taking into account~\cite[Lemmas 8.2 and 8.3]{MRSY$_2$}, passing to
the covering of~${\Bbb B}^n/G$ with a finite or countable number of
normal neighborhoods, which exists in view of Remark~\ref{rem1}, and
using the countable semi-additivity of the measure~$\widetilde{h},$
we obtain the following statement.

\medskip
\begin{proposition}\label{pr1}
{\sl\, Let $f:D\rightarrow {\Bbb B}^n/G_*$ be a mapping that is
almost everywhere differentiable in local coordinates and, moreover,
has $N$ and $N^{\,-1}$-Luzin properties. Then there exists an at
most countable sequence of compact sets $C_k^{\,*}\subset D,$ such
that $\widetilde{h}(B)=0,$ $B=D\setminus
\bigcup\limits_{k=1}^{\infty} C_k^{\,*},$ and $f|_{C_k^{\,*}}$ is
one-to-one and bilipschitz in local coordinates for each
$k=1,2,\ldots .$ Moreover, $f$ is differentiable on $C_k^{\,*},$
wherein $J(x, f)\ne 0$ for $x\in C_k^{\,*}.$}
\end{proposition}

\medskip
Let $\gamma:[a, b]\rightarrow {\Bbb B}^n/G$ be a locally rectifiable
path on ${\Bbb B}^n/G.$ Define the function $l_{\gamma}(t)$ as the
length of the path $\gamma|_{[a, t]},$ $a\leqslant t\leqslant b,$
where ''length'' is understood in the sense of ${\Bbb B}^n/G.$ Given
$B\subset {\Bbb B}^n/G,$ we set
\begin{equation}\label{eq36}
l_{\gamma}(B)={\rm mes}_1\,\{s\in [0, l(\gamma)]: \gamma(s)\in
B\}\,,
\end{equation}
where, as usual, ${\rm mes}_1$ denotes the Lebesgue linear measure
in ${\Bbb R},$ and $l(\gamma)$ is the length of~$\gamma.$ Similarly,
we may define $l_{\gamma}(B)$ for the dashed line $\gamma,$ that is,
when $\gamma:~\bigcup\limits_{i=1}^{\infty}(a_i, b_i)\rightarrow
{\Bbb B}^n/G,$ where $a_i<b_i$ for all $i\in {\Bbb N}$ and $(a_i,
b_i)\cap (a_j, b_j)=\varnothing$ for $i\ne j.$ Hereinafter, the set
$E\subset {\Bbb B}^n/G$ will be called {\it measurable,} if $E$ may
be covered with a countable number of open sets $U_k,$ $k=1,2,\ldots
,$ homeomorphic to the unit ball by some mapping $\varphi_k: U_k
\rightarrow D $ so that $\varphi_k(U_k\cap E)$ is measurable
relative to the Lebesgue measure in ${\Bbb R}^n.$ Similarly, one can
define {\it Borel sets} $E\subset{\Bbb B}^n/G.$ We now prove the
following statement; see also~\cite[Theorem~33.1]{Va$_3$}.

\medskip
\begin{lemma}\label{lem1}
{\sl\, Suppose that the set $B_0\subset {\Bbb B}^n/G$ has zero
$\widetilde{h}$-measure. Then
\begin{equation}\label{eq18}
l_{\gamma}(B_0)=0\,.
\end{equation}
for almost all paths $\gamma$ in ${\Bbb B}^n/G.$
}
\end{lemma}

\begin{proof}
Since the Lebesgue measure is regular, there exists a Borel set
$B\subset {\Bbb B}^n/G$ such that $B_0\subset B$ and
$\widetilde{h}(B_0)=\widetilde{B}=0,$ where $\widetilde{h}$ is a
measure on ${\Bbb B}^n/G,$ defined by~(\ref{eq2B}). Based on the
comments given before Proposition~\ref{pr1}, there is a sequence of
points $x_i,$ $i=1,2,\ldots ,,$ and the corresponding radii of the
balls $r_i=r_i(x_i)>0,$ $i=1,2,\ldots ,,$ such that
$${\Bbb B}^n/G=\bigcup\limits_{i=1}^{\infty} \widetilde{B}(x_i, r_i)\,,\quad
\overline{\widetilde{B}}(x_i, r_i)\subset U_i\,.$$
Denote by $\varphi_i=\pi_i^{\,-1}$ the map corresponding to the
definition of the normal neighborhood $U_i$ (see the comments made
before Proposition~\ref{pr1}). Let $g_i$ be the characteristic
function of the set $\varphi_i(B\cap U_i).$ By~\cite[Theorem~3.2.5
for $m=1$]{Fe}, we obtain that
\begin{equation}\label{eq19}
\int\limits_{\varphi_i(\gamma)}g_i(z)\,|dz|={\mathcal
H}^{\,1}(\varphi_i(B\cap |\gamma|))\,,
\end{equation}
where $\gamma:[a, b]\rightarrow {\Bbb B}^n/G$ -- is an arbitrary
locally rectifiable path, $|\gamma|$ is the locus of $\gamma$ in
${\Bbb B}^n,$ and $|dz|$ is an element of Euclidean length. Arguing
as in the proof of~\cite[Theorem~33.1]{Va$_3$}, we set
$$\rho(p)= \left \{\begin{array}{rr}\infty, & p\in B,
\\ 0 ,  &  p \notin B\ .
\end{array} \right.$$
Note that $\rho$ is a Borel function. Let $\Gamma_i$ be the
subfamily of all paths from $\Gamma,$ for which ${\mathcal
H}^{\,1}(\varphi_i(B\cap |\gamma|))>0.$ By~(\ref{eq19}), for
$\gamma\in\Gamma_i,$ we obtain that
$$\int\limits_{\gamma\cap \widetilde{B}(x_i, r_i)}\rho(p)\,ds_{\widetilde{h}}(p)=\int\limits_{\varphi_i(\gamma)}
\rho(\pi_i(y))\,ds_h(y)=2\int\limits_{\varphi_i(\gamma)}\frac{\rho(\pi_i(y))}{(1-|y|^2)^n}|dy|=$$
$$=2\int\limits_{\varphi_i(\gamma)}\frac{g(y)\rho(\pi_i(y))}{(1-|y|^2)^n}|dy|=\infty\,,$$
where $\gamma\cap \widetilde{B}=\gamma|_{S_i}$ is a dished line,
$S_i=\{s\in [0, l(\gamma)]: \gamma(s)\in \widetilde{B}(x_i, r_i)\}.$
Thus, $\rho\in{\rm adm}\,\Gamma_i$ and
\begin{equation}\label{eq5A}
M(\Gamma_i)\leqslant \int\limits_{{\Bbb
B}^n/G}\rho^n(p)\,d\widetilde{h}(p)=0\,.
\end{equation}
Note that
$\Gamma>\bigcup\limits_{i=1}^{\infty}\Gamma_i,$ so taking into
account the relation~(\ref{eq5A}), we obtain that
$M(\Gamma)\leqslant\sum\limits_{i=1}^{\infty}M(\Gamma_i)=0.$ Lemma
is proved.~$\Box$
\end{proof}

\medskip
Let $f:D\rightarrow {\Bbb R}^n$ (or $f:D\rightarrow {\Bbb B}^n/G$)
is a map for which the image of any curve in $D$ does not degenerate
into a point. Let $I_0$ be a segment and $\beta: I_0\rightarrow
{\Bbb R}^n$ (or $\beta: I_0\rightarrow {\Bbb B}^n/G$) be a
rectifiable path. Let also $\alpha: I\rightarrow D$ be some path
such that $f\circ \alpha\subset \beta.$  If a function of length
$l_{\beta}: I_0\rightarrow [0, l(\beta)]$ is constant on some
interval $J\subset I,$ then $\beta$ is constant on $J$ and, due to
the assumption on $f,$ the path $\alpha$ is also constant on $J.$
This implies that there is only path $\alpha^{\,*}:
l_\beta(I)\rightarrow D,$ such that $\alpha=\alpha^{\,*}\circ
(l_\beta|_I).$ We say that $\alpha^{\,*}$ is {\it $f$-representation
of $\alpha$ with respect to~$\beta.$}

\medskip
The following lemma holds.

\medskip
\begin{lemma}\label{lem7}
{\sl\, Suppose that $G$ and $G_*$ are two groups of M\"{o}bius
transformations of the unit ball ${\Bbb B}^n,$ $n\geqslant 2,$ onto
itself, acting discontinuously on ${\Bbb B}^n$ and not having fixed
points in ${\Bbb B}^n.$ Let $D$ be a domain in ${\Bbb B}^n/G$ and
$f: D\rightarrow{\Bbb B}^n/G_*$ is a discrete map with
$ACP^{\,-1}$-property. Denote by $\gamma^{\,*}$ the
$f$-representation of $\gamma$ with respect to
$\widetilde{\gamma}=f\circ\gamma.$ Then $\gamma^{\,*}$ is absolutely
continuous in local coordinates for almost all closed paths
$\widetilde{\gamma}$ in ${\Bbb B}^n/G_*.$ }
\end{lemma}

\medskip
\begin{proof}
Indeed, $\gamma^{\,*}$ is rectifiable for almost all
$\widetilde{\gamma}$ by $f\in ACP^{\,-1}$-property. Let $L_{\gamma,
f}^{\,-1}$ be a function that participates in the definition of
$ACP^{\,-1}$-property, where the corresponding functions of length
are understood in the sense of the factor spaces under
consideration. Then
$$\gamma^{\,*}\circ l_{\widetilde{\gamma}}(t)=\gamma(t)=\gamma^{\,0}\circ
l_{\gamma}(t)=\gamma^{\,0}\circ L_{\gamma,
f}^{\,-1}\left(l_{\widetilde{\gamma}}(t)\right)$$
Denoting $s:=l_{\widetilde{\gamma}}(t),$ we obtain that
\begin{equation}\label{eq15}\gamma^{\,*}(s)=\gamma^{\,0}\circ L_{\gamma, f}^{\,-1}(s)\,.
\end{equation}
Then $\gamma^{\,*}$ is absolutely continuous in local coordinates,
since by the condition of the lemma $L_{\gamma, f}^{\,-1}(s)$ is
absolutely continuous and
$$\widetilde{h}(\gamma^{\,0}(s_1), \gamma^{\,0}(s_2))\leqslant |s_1-s_2|$$
for all $s_1, s_2\in [0, l(\gamma)].$ Here we also take into account
that locally $\widetilde{h}(\gamma^{\,0}(s_1), \gamma^{\,0}(s_2))$
coincides with $h(\varphi(\gamma^{\,0}(s_1)),
\varphi(\gamma^{\,0}(s_2)))$ in the corresponding local coordinates
$(U, \varphi).$ In addition, by Lemma~\ref{lem3},
$|\varphi(\gamma^{\,0}(s_1))- \varphi(\gamma^{\,0}(s_2))|\leqslant
h(\varphi(\gamma^{\,0}(s_1)), \varphi(\gamma^{\,0}(s_2))).$
\end{proof}~$\Box$

\medskip
\section{Proof of Theorem~\ref{th1}} Let $B_f$ be a branch set of $f.$ Note
that $\widetilde{h}(B_f)=0$ because $f$ is open and discrete,
see~\cite[Proposition~8.4]{MRSY$_2$}. Let $B$ and $C_k^*,$
$k=1,2,\ldots ,$ be sets corresponding to the notation of
Proposition~\ref{pr1}. Setting $B_1=C_1^*,$ $B_2=C_2^*\setminus
B_1,\ldots ,$
\begin{equation} \label{eq7.3.7y} B_k=C_k^*\setminus
\bigcup\limits_{l=1}\limits^{k-1}B_l\,,
\end{equation}
we obtain a countable cover of $D$ by pairwise disjoint sets $B_k,
k=0,1,2,\ldots $ such that $\widetilde{h}(B_0)=0,$ $B_0=B\cup
B_f=\left(D\setminus \bigcup\limits_{k=1}^{\infty} B_k\right)\bigcup
B_f.$ Since, by assumption, the map $f$ has $N$-property in $D,$
then $\widetilde{h_*}(f(B_0))=0.$

Since $\overline{D}$ and $\overline{D_*}$ are compact sets, there
are finite coverings $U_i,$ $1\leqslant i\leqslant I_0$ and $V_m,$
$1\leqslant m\leqslant N_0,$ such that
$$\overline{D}\subset\bigcup\limits_{i=1}^{I_0}U_i\,,\quad \overline{D_*}\subset\bigcup
\limits_{m=1}^{N_0}V_m\,,$$
where $U_i$ and $V_m$ are normal neighborhoods of some points
$x_i\in {\Bbb B}^n/G$ and $y_m\in {\Bbb B}^n/G_*.$ We may choose
these covers so that $\widetilde{h}(\partial
U_i)=\widetilde{h_*}(\partial V_m)=0$ for every $1\leqslant
i\leqslant I_0$ and $V_m,$ $1\leqslant m\leqslant N_0.$ In
particular, there are mappings $\varphi_i: U_i\rightarrow B(0,
r_i),$ $0<r_i<1,$ and $\psi_m: V_m\rightarrow B(0, R_m),$ $0<R_m<1,$
such that the volume and length in $U_i$ and $V_m$ are calculated
using the maps $\varphi_i$ and $\psi_m$ according to
formulas~(\ref{eq2B}) and~(\ref{eq5D}). Set
\begin{equation}\label{eq13}
R_0:=\max\limits_{1\leqslant m\leqslant N_0}R_m\,,\qquad
r_0:=\max\limits_{1\leqslant i\leqslant I_0}r_i\,.
\end{equation}
We now set $U_1^{\,\prime}=U_1,$ $U^{\,\prime}_2=U_2\setminus
\overline{U_1},$ $U^{\,\prime}_3=U_3\setminus (\overline{U_1}\cup
\overline{U_2}),$ $\ldots, U^{\,\prime}_{I_0}=U_{I_0}\setminus
(\overline{U_1}\cup \overline{U_2}\ldots \overline{U_{I_0-1}}).$
Note that by definition $U^{\,\prime}_i\subset U_i$ for $1\leqslant
i\leqslant I_0$ and $U^{\,\prime}_i\cap U^{\,\prime}_j=\varnothing$
for $i\ne j.$ In addition,
$D=\left(\bigcup\limits_{i=1}^{I_0}U^{\,\prime}_i\right)\bigcup
B_0^{\,*},$ where $U^{\,\prime}_i$ are open, and
$\widetilde{h}(B_0^{\,*})=0.$

\medskip
Similarly, we set $V_1^{\,\prime}=V_1,$ $V^{\,\prime}_2=V_2\setminus
\overline{V_1},$ $V^{\,\prime}_3=V_3\setminus (\overline{V_1}\cup
\overline{V_2}),$ $\ldots, V^{\,\prime}_{N_0}=V_{N_0}\setminus
(\overline{V_1}\cup \overline{V_2}\ldots \overline{V_{N_0-1}}).$ By
definition, $V^{\,\prime}_m\subset V_m$ for $1\leqslant m\leqslant
N_0$ and $V^{\,\prime}_m\cap V^{\,\prime}_j=\varnothing$ for $m\ne
j.$ Again, we have that
$D_*=\left(\bigcup\limits_{m=1}^{N_0}V^{\,\prime}_m\right)\bigcup
B_0^{\,**},$ while $V^{\,\prime}_m$ are open, and
$\widetilde{h_*}(B_0^{\,**})=0.$

\medskip
Next we set $U_{m, i}=f^{\,-1}(V^{\,\prime}_m)\cap U^{\,\prime}_i.$
Note that, by construction and continuity of $f,$ the sets $U_{m,
i}$ are open, in addition, $\widetilde{h}(f^{\,-1}(B_0^{\,**}))=0$
by $N^{\,-1}$-property of $f.$ Thus,
\begin{equation}\label{eq7A}
D\subset \left(\bigcup\limits_{1\leqslant i\leqslant I_0\atop
1\leqslant m\leqslant N_0}U_{m, i}\right)\bigcup
f^{\,-1}(B_0^{\,**})\bigcup B_0^{\,*}\,,
\end{equation}
see Figure~\ref{fig1} for an illustration.
\begin{figure}[h]
\centerline{\includegraphics[scale=0.5]{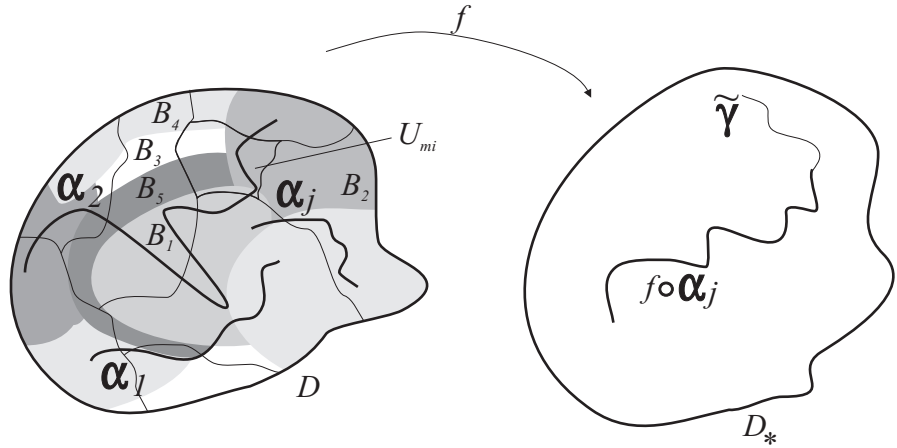}} \caption{To
the proof of Theorem~\ref{th1}}\label{fig1}
\end{figure}
Note that the equality $U_{m_1i_1}=U_{m_2i_2}$ is possible only with
$m_1=m_2$ and $i_1=i_2.$ Indeed, let $x\in U_{m_1i_1}\cap
U_{m_2i_2}.$ Then, in particular, $x\in U^{\,\prime}_{i_1}\cap
U^{\,\prime}_{i_2},$ which is possible only for $i_1=i_2,$ because
$U^{\,\prime}_i\cap U^{\,\prime}_j=\varnothing$ for $i\ne j.$
Further, the condition  $x\in U_{m_1i_1}\cap U_{m_2i_2}$ implies
that $f(x)\in V^{\,\prime}_{m_1}\cap V^{\,\prime}_{m_2},$ which is
also impossible with $m_1\ne m_2,$ because $V^{\,\prime}_i\cap
V^{\,\prime}_j=\varnothing$ for $i\ne j.$ Thus, $i_1=i_2$ and
$m_1=m_2,$ as required.

Set
$$f_{m, i}(x):=(\psi_m\circ f\circ\varphi_i^{\,-1})(x)\,,\quad x=\varphi_i(p)\,, \quad p\in U_{m, i}\,.$$
Let $\rho\in{\rm adm}\,\Gamma $ and
\begin{equation}\label{eq7.3.13}
\widetilde{\rho}(p_*)\quad=\quad\frac{1}{\widetilde{m}}
\chi_{f\left(D\setminus B_0
\right)}(p_*)\sup\limits_{C}\,\sum\limits_{p\,\in\,C}\rho^{\,*}(p)\,,
\end{equation}
where $C$ runs over all subsets of the form $f^{\,-1}(p_*)$ in
$D\setminus B_0,$ whose number of elements is not more than
$\widetilde{m},$ besides,
%
$$\rho^{\,*}(p)= \left \{\begin{array}{rr}\rho
(p)/l(f_{m, i}^{\,\prime}(\varphi_i(p))),
&  \text{при }\quad p\in U_{m, i}\setminus B_0, \\
0, & \text{otherwise.}\end{array} \right.$$
%
%

Note that
$$\widetilde{\rho}(p_*)=\frac{1}{\widetilde{m}}\sup\sum\limits_{l=1}^s\rho_{k_l,
i_l, m_l}(p_*)\,,$$ where
$k_l\in{\Bbb N},$ $1\leqslant i_l\leqslant I_0,$ $1\leqslant
m_l\leqslant N_0,$ $k_l\ne k_s,$ $i_l\ne i_s$ and $m_l\ne m_s$ for
$l\ne s,$

$$\rho_{k, i, m}(p_*)= \left \{\begin{array}{rr}
\rho^*(f^{-1}_{k, i, m}(p_*)), &  {\rm при }\,\,\,  p_*\in f(B_k\cap U_{m, i}), \\
0, & \text{otherwise,}\end{array} \right.
$$
moreover, the mapping $f_{k, i, m}=f|_{B_k\cap U_{m, i}},$
$k=1,2,\ldots,$ is injective. It follows that the function
$\widetilde{\rho}$ is Borel, see~\cite[section~2.3.2]{Fe}.

\medskip
We first consider the case when $\Gamma^{\,\prime}$ consists of
closed paths. Without loss of generality, we may assume that all
paths of the family $\Gamma^{\,\prime}$ a rectifiable. Let
$\widetilde{\gamma}\in\Gamma^{\,\prime}$ and
$\widetilde{\gamma}^0\colon [0, l(\widetilde{\gamma})]\rightarrow
{\Bbb B}^n/G_*$ be a normal representation of a path
$\widetilde{\gamma},$
$\widetilde{\gamma}(t)=\widetilde{\gamma}^0\circ
l_{\widetilde{\gamma}}(t).$ Let $I_j$ be a definition interval for
$\alpha_j^{\,*},$ which is the corresponding $f$-representation of
$\alpha_j$ with respect to $\widetilde{\gamma},$ i.e.
$\alpha_j(t)=\alpha^*_j\circ l_{\widetilde{\gamma}}(t)$ and $f\circ
\alpha^*_j\subset \widetilde{\gamma}^0.$ Paths $\alpha_j\in \Gamma$
with the specified property exist by the condition of the theorem,
and the number of these paths is~$\widetilde{m}.$ We denote
$$
h_j(s)=\rho^{\,*}\left(\alpha^*_j(s)\right)\chi_{I_j}(s)\,,\quad
s\in [0, l(\widetilde{\gamma})]\,,\quad J_s:=\{j\,|\,s\in I_j\}\,.
 $$
Note that the set
$$S_{m, i}=\{s\in [0, l(\widetilde{\gamma})]:
\widetilde{\gamma}^0(s)\in f(U_{m, i})\}$$
is open in $[0, l(\widetilde{\gamma})]$ as a preimage of an open set
$U_{m_i}$ under continuous mapping~$\gamma^*.$ Thus,
$\widetilde{\gamma}|_{S_{m, i}}$ is no more than a countable number
of open arcs, the length of each of which is calculated in
coordinates $(V^{\,\prime}_m, \psi_m)$ with the help of a hyperbolic
metrics (see notes made in the introduction). Denote
$\widetilde{\gamma}_{m, i}:=\widetilde{\gamma}|_{S_{m, i}}.$
According to the above, $\widetilde{\gamma}_{m,
i}=\bigcup\limits_{l=1}^{\infty}\widetilde{\gamma}^l_{m, i},$ where
$\widetilde{\gamma}^l_{m, i}$ is some open arc. (Since the path
$\widetilde{\gamma}$ was chosen to be closed, exactly two of the
indicated arcs may turn out to be half-open, however, we interpret
intervals of the form $[a, c)$ and $(c, b]$ as open sets with
respect to the segment $[a, b ]$). Since $f$ has $N$-property,
$\widetilde{h_*}(B_0^{\,**}\cup f(B_0^{\,*}))=0.$ Then, by
Lemma~\ref{lem1}, $[0,
l(\widetilde{\gamma})]=\bigcup\limits_{1\leqslant i\leqslant
I_0\atop 1\leqslant m\leqslant N_0}S_{m, i}\cup B_*,$ where $B_*$ is
some set of zero linear measure. In this case,
\begin{equation}\label{eq6}
\int\limits_{\widetilde{\gamma}}
\widetilde{\rho}(p_*)\,ds_{\widetilde{h_*}}(p_*)=\sum\limits_{1\leqslant
i\leqslant I_0\atop 1\leqslant m\leqslant N_0}\int\limits_{S_{m, i}}
\widetilde{\rho}(\widetilde{\gamma}^0(s))\,ds
\end{equation}
for almost all $\widetilde{\gamma}\in \Gamma^{\,\prime}.$ Since
$\widetilde{h_*}(f(B_0))=0,$ by Lemma~\ref{lem1}
$\widetilde{\gamma}^0(s)\not\in f(B_0)$ for almost all $s\in [0,
l(\widetilde{\gamma})]$ and almost all $\widetilde{\gamma}\in
\Gamma^{\,\prime}.$ Then for the same $s\in [0,
l(\widetilde{\gamma})]$ points $\alpha^{\,*}_j(s)\in
f^{\,-1}(\widetilde{\gamma}^0(s))$ corresponding to different $j\in
J_s$ are different, because $\alpha_j(s)=p$ is no more than for
$i(p, f)$ indices $j.$ Then, by the definition of the function
$\widetilde{\rho}$ in~(\ref{eq7.3.13})
\begin{equation}\label{eq6.1}
\widetilde{\rho}(\widetilde{\gamma}^0(s))\geqslant
\frac{1}{\widetilde{m}}\sum\limits_{j=1}^{\widetilde{m}} h_j(s)
\end{equation}
for almost all $s\in [0, l(\widetilde{\gamma})].$

By~(\ref{eq6.1}), we obtain that
$$
\int\limits_{S_{m, i}} \widetilde{\rho}(\widetilde{\gamma}^0(s))\,ds
\geqslant \frac{1}{\widetilde{m}}\sum\limits_{j=1}^{\widetilde{m}}
\int\limits_{S_{m, i}}h_j(s)\, ds=
$$
\begin{equation}\label{eq2*}
=
\frac{1}{\widetilde{m}}\sum\limits_{j=1}^{\widetilde{m}}\int\limits_{S_{m,
i,
j}}\rho^{\,*}\left(\alpha^*_j(s)\right)ds=\frac{1}{\widetilde{m}}\sum\limits_{j=1}^{\widetilde{m}}\int\limits_{S_{m,
i, j}} \frac{\rho(\alpha_j^{\,*}(s))}{l(f_{m,
i}^{\,\prime}(\varphi_i(\alpha_j^{\,*}(s))))}\,ds\,,
\end{equation}
for almost all $\widetilde{\gamma}$ and all the corresponding
$\alpha_j,$ $j=1,2,\ldots, \widetilde{m},$
where
$$\alpha_j:I_j\rightarrow {\Bbb B}^n/G,\qquad S_{m,
i, j}=I_j\cap S_{m, i}\,.$$
Since ${\widetilde{\gamma}}$ is rectifiable,
${\widetilde{\gamma}}^{\,0}$ is also rectifiable, in particular,
${\widetilde{\gamma}}^{\,0}(s)$ is almost everywhere differentiable
(see Lemma~\ref{lem2}). By Lemma~\ref{lem7}, each  $\alpha_j^{\,*},$
$j=1,2,\ldots, \widetilde{m},$ is absolutely continuous for almost
all~$\widetilde{\gamma}\in\Gamma^{\,\prime}.$

Since $\widetilde{\gamma}^0(s)\not\in f(B_0)$ for almost all $s\in
[0, l(\widetilde{\gamma})]$ and almost all paths
$\widetilde{\gamma},$ then $\alpha_j^{\,*}(s)\not\in B_0$ for almost
all $s\in S_{m, i, j}.$ Therefore, $\left(f_{m,
i}\left(\varphi_i(\alpha_j^{\,*}(s))\right)\right)^{\,\prime}$ and
$\left(\varphi_i(\alpha_j^{\,*}(s))\right)^{\,\prime}$ exist for
almost all $s\in S_{m, i, j}$ and every $1\leqslant i\leqslant I_0,$
$1\leqslant m\leqslant N_0$ and $1\leqslant
j\leqslant\widetilde{m}.$  Recall that $\widetilde{\gamma}_{m,
i}=\bigcup\limits_{l=1}^{\infty}\widetilde{\gamma}^l_{m, i},$ where
$\widetilde{\gamma}^l_{m, i}:=\widetilde{\gamma}|_{\Delta^l_{m, i}}$
and either $\Delta^l_{m, i}=(\alpha^l_{m, i}, \beta^l_{m, i}),$ or
$\Delta^l_{m, i}=[\alpha^l_{m, i}, \beta^l_{m, i}),$ or
$\Delta^l_{m, i}=(\alpha^l_{m, i}, \beta^l_{m, i}].$ Note that
\begin{equation}\label{eq11}
l_{\widetilde{\gamma}}(s)=\alpha^l_{m, i}+s_h(s)\qquad
\forall\,\,s\in \Delta^l_{m, i}\,, l=1,2,\ldots\,,
\end{equation}
where $s_h(s)$ denotes the hyperbolic length of the path
$\psi_m(\widetilde{\gamma}_{\Delta^l_{m, i}})$ on the segment
$[\alpha^l_{m, i}, s].$ By~(\ref{eq11}) and Lemma~\ref{lem2}, we
obtain that
\begin{equation}\label{eq12}
\left|\frac{d}{ds}\left(f_{m,
i}\left(\varphi_i(\alpha_j^{\,*}(s))\right)\right)\right|=\frac{1-|f_{m,
i}\left(\varphi_i(\alpha_j^{\,*}(s))\right)|^2}{2}\leqslant
\frac{1}{2}
\end{equation}
for almost all $s\in \Delta^l_{m, i}.$ On the other hand, according
to the rule of the derivative of the superposition of functions,
$$\left|\frac{d}{ds}\left(f_{m,
i}\left(\varphi_i(\alpha_j^{\,*}(s))\right)\right)\right|=$$
\begin{equation}\label{eq2A}
=|f_{m, i}^{\,\prime}(\varphi_i(\alpha_j^{\,*}(s)))\cdot
(\varphi_i(\alpha_j^{\,*}(s)))^{\,\prime}| =\left|f_{m,
i}^{\,\prime}(\varphi_i(\alpha_j^{\,*}(s)))\cdot\frac{(\varphi_i(\alpha_j^{\,*}(s)))^{\,\prime}}
{|(\varphi_i(\alpha_j^{\,*}(s)))^{\,\prime}|}\right|\cdot
|(\varphi_i(\alpha_j^{\,*}(s)))^{\,\prime}|\geqslant
\end{equation}
$$\geqslant l(f_{m,
i}^{\,\prime}(\varphi_i(\alpha_j^{\,*}(s))))\cdot|(\varphi_i(\alpha_j^{\,*}(s)))^{\,\prime}|
$$
for almost all ~$s\in \Delta^l_{m, i}\cap I_j.$
Combining~(\ref{eq12}) and (\ref{eq2A}), we obtain that
\begin{equation}\label{eq3A}
\frac{\rho(\alpha_j^{\,*}(s))}{l(f_{m,
i}^{\,\prime}(\varphi_i(\alpha_j^{\,*}(s))))}\geqslant
2\rho(\alpha_j^{\,*}(s))\cdot
|(\varphi_i(\alpha^{\,*}_j(s)))^{\,\prime}|\,.
\end{equation}
for almost all~$s\in S_{m, i}.$

Let $\alpha_j^0$ be a normal representation of $\alpha_j.$ Then,
since $f\in ACP^{\,-1},$ $\alpha_j^0(s_0)\not\in
f^{\,-1}(B_0^{\,**})\bigcup B_0^{\,*}$ for almost all $s_0\in [0,
l(\alpha_j)]$ and almost all $\widetilde{\gamma}=f\circ \alpha_j$
(see~\cite[Theorem~2.10.13]{Fe}). Denote
$$Q_{m, i, j}=\{s_0\in [0, l(\alpha_j)]:
\alpha_j(s_0)\in U_{m, i}\}\,.$$

By absolute continuity of $\alpha_j^{\,*}(s),$ $j=1,2,\ldots,
\widetilde{m},$ and also due to~(\ref{eq14}), (\ref{eq15}) and
(\ref{eq7A}), we obtain, that
$$1\leqslant
\int\limits_{\alpha_j}\rho(p)\,ds_{\widetilde{h}}(p)=\sum\limits_{1\leqslant
i\leqslant I_0\atop 1\leqslant m\leqslant N_0}\int\limits_{Q_{m, i,
j}}\rho(\alpha_j^0(s_0))\,ds_0=
$$
\begin{equation}\label{eq4A}
=\sum\limits_{1\leqslant i\leqslant I_0\atop 1\leqslant m\leqslant
N_0}\int\limits_{S_{m,
i}}\frac{2\rho(\alpha_j^{\,*}(s))|(\varphi_i(\alpha^{\,*}_j(s)))^{\,\prime}|}{1-|\varphi_i(\alpha_j^{\,*}(s))|^2}\,ds\leqslant\frac{2}{1-r_0^2}
\sum\limits_{1\leqslant i\leqslant I_0\atop 1\leqslant m\leqslant
N_0}\int\limits_{S_{m, i,
j}}\rho(\alpha_j^{\,*}(s))|(\varphi_i(\alpha^{\,*}_j(s)))^{\,\prime}|\,ds\,.
\end{equation}
Combining~(\ref{eq6}), (\ref{eq2*}), (\ref{eq3A}) and (\ref{eq4A}),
we conclude that
$\int\limits_{\widetilde{\gamma}}\widetilde{\rho}(p_*)\,ds_{\widetilde{h_*}}(p_*)\geqslant
1$ for almost all closed paths $\widetilde{\gamma}\in
\Gamma^{\,\prime}.$ The case of arbitrarily (not necessarily closed
$\widetilde{\gamma}$ may be obtained by taking $\sup$ in the
expression
$\int\limits_{\widetilde{\gamma}^{\,\prime}}\widetilde{\rho}(p)\,ds_{\widetilde{h_*}}(p_*)\geqslant
1$ over all closed sub-paths of $\widetilde{\gamma}^{\,\prime}.$
Therefore, $\frac{1}{1-r_0^2}\cdot\widetilde{\rho}\in{\rm
adm}\,\Gamma^{\,\prime}.$ So,
\begin{equation}\label{eq3*}
M\left(f\left(\Gamma\right)\right)\leqslant
\frac{1}{(1-r_0^2)^n}\int\limits_{D_*}
{\widetilde{\rho}}^{\,n}(p_*)\, dh_*(p_*)\,.
\end{equation}

\medskip
By~\cite[Theorem~3.2.5, $m=n$]{Fe}, we obtain that
$$\int\limits_{U_{m, i}\cap
B_k} K_I(p,
f)\cdot\rho^n(p)\,d\widetilde{h}(p)=$$$$=2^n\int\limits_{\varphi_i(U_{m,
i}\cap B_k)} \frac{\Vert(\psi_m\circ f\circ
\varphi^{\,-1}_i)^{\,\prime}(x)\Vert^n}{\det\{(\psi_m\circ f\circ
\varphi^{\,-1}_i)^{\,\prime}(x)\}(1-|x|^2)^n}\cdot\rho^n(\varphi^{\,-1}_i(x))\,dm(x)\geqslant$$
$$\geqslant 2^n\int\limits_{\varphi_i(U_{m,
i}\cap B_k)} \frac{\Vert(\psi_m\circ f\circ
\varphi^{\,-1}_i)^{\,\prime}(x)\Vert^n}{\det\{(\psi_m\circ f\circ
\varphi^{\,-1}_i)^{\,\prime}(x)\}}\cdot\rho^n(\varphi^{\,-1}_i(x))\,dm(x)=$$
\begin{equation} \label{eq7.3.14}
=2^n\int\limits_{\psi(f((U_{m, i}\cap
B_k)))}\frac{\rho^{\,n}\left((f_k^{\,-1}\circ\psi_m^{\,-1})(y)\right)}
{\left\{l\left(f^{\,\prime}\left((\varphi_i\circ
f^{\,-1}_k\circ\psi^{\,-1}_m)(y)\right)\right)\right\}^n}\,dm(y)\geqslant
\end{equation}
$$\geqslant (1-R_0^2)^n\int\limits_{f(D)}\rho^{n}_{k, i, m}(p_*)\,d\widetilde{h_*}(p_*)\,.$$
Finally, by Lebesgue's theorem, see~\cite[Theorem~I.12.3]{Sa}, and
taking into account~(\ref{eq3*}) and~(\ref{eq7.3.14}), we obtain
that
$$\frac{1}{\widetilde{m}}\cdot\int\limits_D K_I(p, f)\cdot\rho^n(p)\,d\widetilde{h}(p)=
\frac{1}{\widetilde{m}}\cdot\sum\limits_{{1\leqslant i\leqslant
I_0,}\,\,{1\leqslant m\leqslant N_0}\atop {1\leqslant k<
\infty}}\int\limits_{U_{m, i}\cap B_k} K_I(p,
f)\cdot\rho^n(p)\,d\widetilde{h}(p)\geqslant$$
$$\geqslant \frac{(1-R_0^2)^n}{\widetilde{m}}
\cdot\int\limits_{f(D)}\sum\limits_{{1\leqslant i\leqslant
I_0,}\,\,{1\leqslant m\leqslant N_0}\atop {1\leqslant k<
\infty}}\rho_{k, i, m}^n(p_*)\,d\widetilde{h_*}(p_*)\geqslant$$
$$\geqslant (1-R_0^2)^n\cdot
\int\limits_{f(D)}\sup \limits_{k_l\in{\Bbb N}, 1\leqslant
i_l\leqslant I_0, 1\leqslant m_l\leqslant N_0 \atop k_l\ne k_s,
i_l\ne i_s, m_l\ne m_s (l\ne s, s\leqslant
m)}\frac{1}{\widetilde{m}}\cdot\sum\limits_{l=1}^s\rho^n_{k_l, i_l,
m_l}(y)\,d\widetilde{h_*}(p_*)\geqslant$$
$$\geqslant (1-R_0^2)^n\cdot
\int\limits_{f(D)}\sup \limits_{k_l\in{\Bbb N}, 1\leqslant
i_l\leqslant I_0, 1\leqslant m_l\leqslant N_0 \atop k_l\ne k_s,
i_l\ne i_s, m_l\ne m_s (l\ne s, s\leqslant
m)}\left(\frac{1}{\widetilde{m}}\cdot\sum\limits_{l=1}^s\rho_{k_l,
i_l, m_l}(y)\right)^n\,d\widetilde{h_*}(p_*)=$$
$$=(1-R_0^2)^n\cdot\int\limits_{f(D)}{\widetilde{\rho}}^{\,n}(p_*)\,d\widetilde{h_*}(p_*)
\geqslant (1-R_0^2)^n(1-r_0^2)^n\cdot M(\Gamma^{\,\prime})\,.$$
In this case, the relation
\begin{equation}\label{eq1B}
M(\Gamma^{\,\prime})\leqslant \frac{c}{\widetilde{m}}\cdot
\int\limits_{D}K_I(p, f)\cdot\rho^n(p)\,d\widetilde{h}(p)
\end{equation}
holds for each $\rho\in {\rm adm\,}\Gamma,$ where
$c:=\frac{1}{(1-R_0^2)^n(1-r_0^2)^n}.$ Passing to the limit
in~(\ref{eq1B}) when $r_0$ and $R_0$ tend to zero, we obtain the
desired relation~(\ref{eq1A}). Theorem~\ref{th1} is proved.~$\Box$

\section{On estimates of the modulus of families of paths in the preimage under the mapping}

As in the case of the space~${\Bbb R}^n,$ in the factor spaces the
lower estimates of the distortion of the module are valid.
Their use is important for the study of mappings, see, for example,
~\cite[Theorem~3.2]{MRV}, 
\cite[Theorem~6.1]{MRSY$_1$}, \cite[Theorems~1.1, 2.1]{SalSev}
and~\cite[Theorem~6.2]{IS}.

\medskip
Let's start with preliminary remarks. For points $x, y\in {\Bbb
B}^n,$ we denote
\begin{equation}\label{eq7B}
L(x, f)=\limsup\limits_{y\rightarrow x}\frac{|f(x)-f(y)|}{|x-y|}\,.
\end{equation}
The following statement was obtained by V\"{a}is\"{a}l\"{a} in the
case of a Euclidean space, see~\cite[Theorem~5.3]{Va$_3$}.

\medskip
\begin{lemma}\label{lem8}
{\sl\,Let $D\subset {\Bbb B}^n,$ let $y_0\in {\Bbb B}^n,$ and let
$\Delta$ be an interval in ${\Bbb R}.$ Let $R_0>0$ be such that
$\overline{B(y_0, R_0)}\subset {\Bbb B}^n.$ Suppose that
$\alpha:\Delta\rightarrow {\Bbb B}^n$ is a locally rectifiable path,
such that the mapping $f:D\rightarrow B(y_0, R_0)$ is absolutely
continuous on any closed subpath of $\alpha.$ Then $f\circ\alpha$ is
locally rectifiable and, if $\rho: |f\circ \alpha|\rightarrow
\overline{\Bbb R}$ is a nonnegative Borel function, then
$$
\int\limits_{f\circ \alpha}\rho(y)\,ds_{h}(y)\leqslant
\frac{1}{1-R_0^2}\cdot\int\limits_{\alpha}\rho(f(x))L(x,
f)\,ds_h(x)\,,
 $$
where $L(x, f)$ is defined in~(\ref{eq7B}). }
\end{lemma}

\medskip
\begin{proof}
By Corollary~\ref{cor2} and~\cite[Theorem~5.3]{Va$_3$}
$$\int\limits_{f\circ \alpha}\rho(y)\,ds_{h}(y)=\int\limits_{f\circ
\alpha}\frac{2\rho(y)}{1-|y|^2}\,|dy|\leqslant
\frac{1}{(1-R_0^2)}\cdot \int\limits_{f\circ
\alpha}2\rho(y)\,|dy|\leqslant$$
\begin{equation}\label{eq20}
\leqslant\frac{1}{(1-R_0^2)}\cdot\int\limits_{\alpha}\rho(f(x))L(x,
f)\,|dx|\leqslant
\frac{1}{1-R_0^2}\cdot\int\limits_{\alpha}\rho(f(x))L(x,
f)\,ds_h(x)\,,
\end{equation}
as required.~$\Box$
\end{proof}

\medskip
In what follows, the {\it inner dilation of the mapping
$f:D\rightarrow{\Bbb R}^n$ at the point $x\in D$} is defined by the
relation
\begin{equation}\label{eq21}
K_O(x,f)\quad =\quad\left\{
\begin{array}{rr}
\frac{\Vert f^{\,\prime}(x)\Vert^n}{|J(x, f)|}, & J(x,f)\ne 0,\\
1,  &  f^{\,\prime}(x)=0, \\
\infty, & \text{otherwise}
\end{array}
\right.\,.
\end{equation}
In the formula above, as usual, it is assumed that the corresponding
values are well-defined for almost all $x\in D$; for example, this
holds if the mapping $f$ is almost everywhere differentiable.

Suppose that $G$ is a group of M\"{o}bius transformations of the
unit ball ${\Bbb B}^n,$ $n\geqslant 2,$ onto itself, acting
discontinuously and not having fixed points in ${\Bbb B}^n.$ If we
are talking about the mapping $f$ of domains $D$ and $D_*,$
belonging to the factor spaces ${\Bbb B}^n/G$ и ${\Bbb B}^n/G_*,$
then we set $K_O(p,f)=K_O(\varphi(x), F),$ where $F=\psi\circ
f\circ\varphi^{-\,1},$ $(U, \varphi)$ are local coordinates of $x$
and $(V, \psi)$ are local coordinates of~$f(x).$ It is easy to show
that, just like $K_I(x, f),$ outher dilatation $K_O(p ,f)$ depends
only on the point $p\in {\Bbb B}^n/G$ and does not depend on local
coordinates. The following assertion holds.

\medskip
\begin{theorem}\label{th2}
{\sl\, Suppose that $G$ and $G_*$ are two groups of M\"{o}bius
transformations of the unit ball ${\Bbb B}^n,$ $n\geqslant 2,$ onto
itself, acting discontinuously on ${\Bbb B}^n$ and not having fixed
points in ${\Bbb B}^n.$ Let $D$ and $D_{\, *}$ be domains belonging
to ${\Bbb B}^n/G$ and ${\Bbb B}^n/G_*,$ respectively, and having
compact closures $\overline{D}$ and $\overline{D_*}.$ Let
$f:D\rightarrow D_{\,*}$ be an open discrete almost everywhere
differentiable map, $f\in ACP,$ having $N$ and $N^{\,-1}$-Luzin
properties. Suppose that $\Gamma$ is a family of paths in $D.$ Then
\begin{equation}\label{eq1D}
M(\Gamma)\leqslant \int\limits_{f(D)}K_I(p_*,
f^{\,-1})\cdot\rho_*^n(p_*)\,d\widetilde{h_*}(p_*)\qquad
\forall\,\,\rho_*\in{\rm adm\,}f(\Gamma)\,,
\end{equation}
where~$K_I(p_*, f^{\,-1}):=\sum\limits_{p\in D\cap f^{\,-1}(p_*)}K_O
(p, f)\,.$}
\end{theorem}

\medskip
\begin{proof} By Remark~\ref{rem1}, $D$ may
be represented in a form of at most countable union of balls of the
form~$\overline{B(x_k, r_k)},$ each of which is compact in ${\Bbb
B}^n/G.$ Thus, $f(D)$ $\sigma$-compact in ${\Bbb B}^n/G_*,$ in
particular, $f(D)$ is Borel.

\medskip
Let $C_k^{\,*},$ $B_k,$ $U_m,$ $V_i$ and $U_{mi}$ correspond to the
notation used in the proof of Theorem~\ref{th1}.

\medskip
Without loss of generality, we may assume that all paths of the
family $\Gamma$ are locally rectifiable. Let $\gamma\in \Gamma$ and
let $\gamma^{\,0}$ be a normal representation of $\gamma.$ By
Lemma~\ref{lem1}, $\gamma^{\,0}(s)\not\in B_0$ for almost all $s\in
[0, l(\gamma)]$ and almost all $\gamma\in \Gamma.$ Since, by the
assumption, $f\in ACP,$ then $f\circ\gamma^{0}$ is locally
rectifiable and absolutely continuous for almost
all~$\gamma\in\Gamma.$
\medskip

As before, for a domain $D_0\subset {\Bbb R}^n$ and a mapping
$\varphi:D_0\rightarrow{\Bbb R}^n$ we use the notation $\Vert
\varphi^{\,\prime}(y)\Vert=\max \{ |f^{\,\prime}(y)h|: h\in{\Bbb
R}^n, |h|=1\}.$ We set
$$f_{m, i}(x):=(\psi_m\circ f\circ\varphi_i^{\,-1})(x)\,,\quad x=\varphi_i(p)\,,\quad \in p\in U_{m, i}\,,$$
$$f_k:=f_{B_k}\,.$$
Let $\rho_*\in{\rm adm}\,f(\Gamma),$ then set
%
$$\rho (p)= \left \{\begin{array}{rr}
\rho_*(f(p))\Vert f_{m, i}^{\,\prime}(\varphi_i(p))\Vert, &  \quad p\in U_{m, i}\setminus B_0, \\
0, & \text{в других случаях.}\end{array} \right. 
$$
Let
$$A_{m, i}=\{s\in [0, l(\gamma)]: s\in U_{m, i}\}\,,\quad |\gamma|\cap U_{m, i}:=\gamma|_{A_{m, i}}\,.$$
Then
\begin{equation}\label{eq22}
\int\limits_{\gamma}
\rho(p)\,ds_{\widetilde{h}}(p)=\sum\limits_{1\leqslant i\leqslant
I_0\atop 1\leqslant m\leqslant N_0}\int\limits_{A_{m,
i}}\rho(\gamma^{\,0}(s))\,ds\,.
\end{equation}
By Lemma~\ref{lem8}
$$\int\limits_{A_{m,
i}}\rho(\gamma^{\,0}(s))\,ds=\int\limits_{A_{m, i}}
\rho_*(f(\gamma^{\,0}(s)))\Vert f_{m,
i}^{\,\prime}(\varphi_i(\gamma^{\,0}(s)))\Vert\,ds\geqslant
$$
\begin{equation} \label{eq7.3.7b}
\geqslant (1-R_0^2)\cdot\int\limits_{\psi_m\circ f\circ\gamma}
(\rho_*\circ\psi^{\,-1}_m)(y)\,ds_h(y)=(1-R_0^2)\cdot\int\limits_{(f\circ\gamma)|_{A_{m,
i}}} \rho_*(y)\,ds_{\widetilde{h}}(y)\,.
\end{equation}
By~(\ref{eq22}) and~(\ref{eq7.3.7b}), taking into account that
$\rho_*\in{\rm adm}\,f(\Gamma),$ we obtain that
\begin{equation}\label{eq23}
\int\limits_{\gamma} \rho(p)\,ds_{\widetilde{h}}(p)\geqslant 1-R_0^2
\end{equation}
for almost all $\gamma\in \Gamma.$
Thus,
\begin{equation} \label{eq7.3.7z} M_p(\Gamma)\leqslant\frac{1}{{(1-R_0^2)}^n}\int\limits_D
\rho^n(p)\, d\widetilde{h}(p)\,. \end{equation}
Note that $\rho=\sum\limits_{k\geqslant 1, 1\leqslant i\leqslant
I_0\atop 1\leqslant m\leqslant M_0}\limits^{\infty}\rho_{k, m, i}$,
where $\rho_{k, m, i}=\rho\cdot\chi_{B_k\cap U_{m, i}}$ are
functions that have pairwise disjoint supports.
By~\cite[Theorem~3.2.5, $m=n$]{Fe}, we obtain that
$$\int\limits_{f(B_k\cap U_{m, i})} K_O\left(f_k^{\,-1}(p_*),
f\right)\cdot\rho_*^n(p_*)\,d\widetilde{h_*}(p_*)=$$$$
2^n\int\limits_{\psi_m(f(B_k\cap U_{m, i}))} K_O\left(f_{k, i,
m}^{\,-1}(y), f_{k, i,
m}\right)\cdot\frac{\rho_*^n(\psi^{\,-1}_m(y))}{{(1-|y|^2)}^n}\,dm(y)\geqslant$$
$$
2^n\int\limits_{\psi_m(f(B_k\cap U_{m, i}))} K_O\left(f_{k, i,
m}^{\,-1}(y), f_{k, i,
m}\right)\cdot\rho_*^n(\psi^{\,-1}_m(y))\,dm(y)=$$
$$
2^n\int\limits_{\varphi_i(B_k\cap U_{m, i})} K_O\left(x, f_{k, i,
m}\right)\cdot\rho_*^n(\varphi_i^{\,-1}(x))|J(x, f_{k, i,
m})|\,dm(x)=$$
$$=2^n\int\limits_{\varphi_i(B_k\cap U_{m, i})}
{\Vert f^{\,\prime}_{k, i, m}(x)\Vert}^n
\cdot\rho_*^n(\varphi_i^{\,-1}(x))\,dm(x)=$$
$$=2^n\int\limits_{\varphi_i(B_k\cap U_{m, i})} {\Vert f^{\,\prime}_{k,
i, m}(x)\Vert}^n
\cdot\rho_*^n(\varphi_i^{\,-1}(x))\cdot\frac{{(1-|x|^2)}^n}{{(1-|x|^2)}^n}\,dm(x)\geqslant
$$
$$\geqslant2^n{(1-r_0^2)}^n\cdot \int\limits_{\varphi_i(B_k\cap
U_{m, i})} {\Vert f^{\,\prime}_{k, i, m}(x)\Vert}^n
\cdot\frac{\rho_*^n(\varphi_i^{\,-1}(x))}{{(1-|x|^2)}^n}\,dm(x)=$$
\begin{equation}\label{eq7.3.7x}={(1-r_0^2)}^n\cdot \int\limits_D \rho^n_{k, m,
i}(p)\,d\widetilde{h}(p)\,.
\end{equation}

\medskip
Finally, by the Lebesgue convergence theorem,
see~\cite[теорема~I.12.3]{Sa}, taking into account~(\ref{eq7.3.7z})
and (\ref{eq7.3.7x}), we obtain that
$$\int\limits_{f(D)} K_I(y, f^{\,-1})\cdot\rho_*^n(p_*)\,d\widetilde{h_*}(p_*)=
{(1-r_0^2)}^n\cdot\int\limits_D \sum\limits_{k\geqslant 1,
1\leqslant i\leqslant I_0\atop 1\leqslant m\leqslant
M_0}\limits^{\infty}\rho_k^n(p)\,d\widetilde{h}(p)\geqslant$$$$\geqslant
{(1-r_0^2)}^n\cdot {(1-R_0^2)}^n\cdot M(\Gamma)\,.$$
From the last relation it follows that
$$M(\Gamma)\leqslant \frac{1}{{(1-r_0^2)}^n{(1-R_0^2)}^n}
\cdot \int\limits_{f(D)} K_I(y, f^{\,-1})\cdot\rho_*^n(p_*)\,d\widetilde{h_*}(p_*)\,.$$
Passing now to the limit at $r_0\rightarrow 0$ and $R_0\rightarrow
0,$ we obtain the desired inequality~(\ref{eq1D}).~$\Box$
\end{proof}

\medskip
\section{Local and boundary behavior of mappings on
factor-spaces}

In conclusion, we consider the application of Theorem~\ref{th1} to
the question of the boundary behavior of maps. Let $D$ be a domain
in ${\Bbb B}^n/G,$ and let $E,$ $F\subset D$ be arbitrary sets. In
what follows, by $\Gamma(E,F, D)$ we denote the family of all paths
$\gamma:[a,b]\rightarrow D$ which join $E$ and $F$ in $D,$ i.e.,
$\gamma(a)\in E,\,\gamma(b)\in F$ and $\gamma(t)\in D$ for
$t\in(a,\,b).$ We agree to say that the boundary $\partial D$ of a
domain $D$ is {\it strongly accessible at $p_0\in
\partial D$}, if for each neighborhood $U$ of $p_0$ there is a
compactum $E\subset D,$ a neighborhood $V\subset U$ of the same
point and the number $\delta
>0$ such that $M(\Gamma(E, F, D))\geqslant \delta$ for any continua
$E$ and $F$ in $D,$ that intersect both $\partial U$ and $\partial
V.$ We also say that the boundary $\partial D$  is {\it strongly
accessible}, if it is strongly achievable at each of its points.
Given $E\subset {\Bbb B}^n/G,$ we set, as usual,
$$C(f, E)=\{p_*\in{\Bbb B}^n/G_*: \exists\,\, p_k\in D, p\in \partial E:
p_k\rightarrow p, f(p_k)\rightarrow p_*, k\rightarrow\infty\}\,.$$
Following~\cite[section~2]{IR} (see
also~\cite[section~6.1]{MRSY$_2$}), we say that the function
${\varphi}:D\rightarrow{\Bbb R}$ has a {\it finite mean oscillation}
at $p_0\in \overline{D},$ abbr. $\varphi\in FMO(p_0),$ if
%
%
%
%
$$\limsup\limits_{\varepsilon\rightarrow
0}\frac{1}{\widetilde{h}(\widetilde{B}(p_0,
\varepsilon))}\int\limits_{\widetilde{B}(p_0,\,\varepsilon)}
|{\varphi}(p)-\overline{\varphi}_{\varepsilon}|\
d\widetilde{h}(p)<\infty\,,$$
%
%
where $\overline{{\varphi}}_{\varepsilon}=\frac{1}
{\widetilde{h}(\widetilde{B}(p_0,
\varepsilon))}\int\limits_{\widetilde{B}(p_0, \varepsilon)}
{\varphi}(p) \,d\widetilde{h}(p).$ The following statement holds.

\medskip
\begin{theorem}\label{th3}{\sl\,  Suppose that $G$ and $G_*$ are two groups of
M\"{o}bius transformations of the unit ball ${\Bbb B}^n,$
$n\geqslant 2,$ onto itself, acting discontinuously on ${\Bbb B}^n$
and not having fixed points in ${\Bbb B}^n.$ Let $D$ and $D_{\, *}$
be domains belonging to ${\Bbb B}^n/G$ and ${\Bbb B}^n/G_*,$
respectively, and having compact closures $\overline{D}$ and
$\overline{D_*}.$ Let $f:D\rightarrow D_{\,*}$ be an open discrete
almost everywhere differentiable map, $f\in ACP^{\,-1},$ having $N$
and $N^{\,-1}$-Luzin properties. Suppose that $D$ is locally path
connected at $b\in \partial D,$ $C(f,
\partial D)\subset \partial D^{\,\prime},$ and $\partial D^{\,\prime}$
is strongly accessible at least at one of the points~$y\in C(f, b).$
If $K_I(p, f)\leqslant Q(p)$ for almost all $p\in D,$ where $Q:{\Bbb
B}^n/G\rightarrow[0,\infty]$ is some function such that $Q\in
FMO(b),$ then $C(f, b)=\{y\}.$
}
\end{theorem}

\medskip
\begin{proof}
Let $$A=A(b, r_1, r_2)=\{p\in {\Bbb B}^n/G: r_1<d(p, p_0)<r_2\},
\quad 0 < r_1 < r_2 <\infty\,.$$
By Theorem~\ref{th1}, $f$ satisfies relation~(\ref{eq1B}) for each
family of paths $\Gamma$ in $D.$ In particular, the condition
\begin{equation}\label{eq15A}
M(f(\Gamma(C_1, C_0, A)))\leqslant\int\limits_{A\cap
D}Q(p)\cdot\rho^n(p)\,d\widetilde{h}(p)\quad\forall\,\rho\in{\rm
adm\,}\Gamma(C_1, C_0, A)\,,
\end{equation}
holds for any two continua~$C_0\subset \overline{\widetilde{B}(b,
r_1)}$ and $C_1\subset {\Bbb B}^n/G\setminus \widetilde{B}(b, r_2).$

Let $\eta:(r_1, r_2)\rightarrow [0, \infty]$ be an arbitrary
Lebesgue measurable function satisfying the condition
$\int\limits_{r_1}^{r_2}\eta(t)\,dt\geqslant 1.$ Set
$\rho(x)=\eta(\widetilde{h}(p, p_0)).$ In this case,
by~\cite[предложение~13.4]{MRSY$_2$},
$\int\limits_{\gamma}\rho(p)\,ds_{\widetilde{h}}(p)\geqslant 1$ for
an arbitrary (locally rectifiable) path $\gamma.$ Thus,
\begin{equation}\label{eq16A}
M(f(\Gamma(C_1, C_0, A)))\leqslant\int\limits_{A\cap
D}Q(p)\cdot\eta({\widetilde{h}}^{n}(p, p_0))\,d\widetilde{h}(p)\,.
\end{equation}
Note that each $\beta:[a, b)\rightarrow D_*$ has a maximal
$f$-lifting in $D$ starting at $x\in f^{\,-1}(\beta(a)),$
see~\cite[Lemma~2.1]{SM}. By Lemma~\ref{lem3}, ${\Bbb B}^n/G$ is
locally Ahlfors regular space. Thus, the necessary conclusion
follows from~\cite[теорема~5]{Sev}. ~$\Box$
\end{proof}

\medskip
Recall that a family $\frak{F}$ of mappings $f\colon X\rightarrow
{X}^{\,\prime}$ is called {\it equicontinuous at a point} $x_0\in
X,$ if for any $\varepsilon>0$ there is $\delta>0$ such that
${d}^{\,\prime} \left(f(x),f(x_0)\right)<\varepsilon$ for all $x\in
D$ with $d(x,x_0)<\delta$ and all $f\in \frak{F}.$  We say, that
$\frak{F}$ is {\it equicontinuous in $D$}, if $\frak{F}$ is
equicontinuous at every point $x_0\in X.$ As one of the possible
applications of Theorem~\ref{th1}, we also give the following
statement.

\medskip
\begin{theorem}\label{theor4*!} {\sl\, Suppose that $G$ and $G_*$ are two groups of
M\"{o}bius transformations of the unit ball ${\Bbb B}^n,$
$n\geqslant 2,$ onto itself, acting discontinuously on ${\Bbb B}^n$
and not having fixed points in ${\Bbb B}^n.$ Let $D$ and $D_{\, *}$
be domains belonging to ${\Bbb B}^n/G$ and ${\Bbb B}^n/G_*,$
respectively, and having compact closures $\overline{D}$ and
$\overline{D_*}.$ Let also $p_0\in D,$ $B_R$ be a ball in $D_*$ and
$Q:D\rightarrow [0, \infty]$ be a function measurable with respect
to the measure $\widetilde{h}.$ Suppose that ${\Bbb B}^n/G_*$ is
$n$-Ahlfors regular space with $(1; n)$-Poincar\'{e} inequality.

Denote by $\frak{R}_{p_0, Q, B_R, \delta}(D)$ the family of all
almost differentiable homeomorphisms $f:D\rightarrow B_R,$ belonging
to the class~$ACP^{\,-1},$ and possessing $N$ and $N^{\,-1}$-Luzin
properties, for which: 1) there is a continuum $K_f\subset B_R
\setminus f(D),$ satisfying the condition $\sup\limits_{x, y\in K_f}
d^{\,\prime}(x, y)\geqslant \delta>0;$ 2) $K_I(p, f)\leqslant Q(p)$
for almost all $p\in D.$ If $Q\in FMO(p_0),$ then the family
$\frak{R}_{p_0, Q, B_R, \delta}(D)$ is equicontinuous at $p_0\in D.$
 }
\end{theorem}

\medskip
{\it Proof} of Theorem~\ref{theor4*!} follows from Theorem~\ref{th1}
and~\cite[Theorem~1]{Sev} together with remarks made in the proof of
Theorem~\ref{th3}.~$\Box$

\medskip
{\bf Example.} As a simple illustration of Theorem~\ref{theor4*!},
consider the following family of mappings. Let $\alpha\geqslant 1$
be a fixed number. Based on the homeomorphism
$h(x)=\frac{x}{|x|}\exp\{\log^{\,\alpha}\frac{1}{|x|}\},$ we put
$$
h_m(x)=\left\{
\begin{array}{rr}
\frac{e\cdot x}{((m-1)/m)\cdot\exp\left\{\log^{\,\alpha}\left(\frac{e}{(m-1)/m}\right)\right\}}, & x\in {\Bbb B}^n\cap B(0, (m-1)/m),\\
\frac{e\cdot
x}{|x|\exp\left\{\log^{\,\alpha}\left(\frac{e}{|x|}\right)\right\}},
&  x\in {\Bbb B}^n\setminus B(0, (m-1)/m)
\end{array}
\right.\,,
$$
where ${\Bbb B}^n=\{x\in {\Bbb R}^n: |x|<1\}.$ It is not difficult
to see that mappings $h_m$ are~$W_{\rm loc}^{1, n}({\Bbb
B}^n)$-homeomorphisms such that~$h_m^{\,-1}\in W_{\rm loc}^{1,
n}({\Bbb B}^n)$ and, therefore, are differentiable almost everywhere
mappings with $ACP^{\,-1},$ $N$- and $N^{\,-1}$-Luzin properties in
${\Bbb B}^n$ (see~\cite[Theorem~28.2]{Va$_3$},
\cite[Corollary~B]{MM} and~\cite[Lemma~3]{Va$_1$}). Using reasoning
given at considering of~\cite[Proposition~6.3]{MRSY$_2$}, it can be
shown that
$K_I(x, h_m)\leqslant \alpha\cdot\log^{\alpha-1}\frac{e}{|x|}$ for
almost all $x\in {\Bbb B}^n.$ If $\alpha=2,$ then $K_I(x,
h_m)\leqslant 2\cdot Q(x),$ where $Q(x):=\log\frac{e}{|x|}\in
FMO(B(0, r_0^{\,\prime}))$ (see~\cite[p.~5]{RR}), where the latter
should be understood with respect to the Euclidean metric and
Lebesgue measure. For a hyperbolic metric and measure, this
statement is also true, since the hyperbolic and Euclidean metrics
are equivalent on compact sets (see Lemma~\ref{lem3}).

Suppose that $G$ is a group of M\"{o}bius transformations of the
unit ball ${\Bbb B}^n,$ $n\geqslant 2,$ onto itself, acting
discontinuously and not having fixed points in ${\Bbb B}^n.$ Suppose
also that ${\Bbb B}^n/G$ is $n$-Ahlfors regular space with $(1;
n)$-Poincar\'{e} inequality. Let $\pi:{\Bbb B}^n\rightarrow {\Bbb
D}/G$ be the natural projection of ${\Bbb B}^n$ onto ${\Bbb B}^n/G,$
and let $p_0\in {\Bbb B}^n/G$ be such that $\pi(0)=p_0.$ Let $r_0>0$
be the radius of a ball with center at a point~$p_0,$ entirely lying
in some normal neighborhood $U$ of $p_0.$ Then by the definition of
the natural projection $\pi,$ as well as the definition of the
hyperbolic metric~$h$ and the metrics~$\widetilde{h}$ in~(\ref{eq2})
we have, that $\pi(B(0, r^{\,\prime}_0))=\widetilde{B}(p_0, r_0),$
where $r^{\,\prime}_0:=(e^{r_0}-1)(e^{r_0}+1).$ In this case, the
family of mappings
$$
\widetilde{g}_m(y)=\left\{
\begin{array}{rr}
\frac{e\cdot
y}{((m-1)/m)\cdot\exp\left\{\log^{\,\alpha}\left(\frac{e}{(m-1)/m}\right)\right\}},
& y\in B(0, r^{\,\prime}_0)\cap B(0, (r^{\,\prime}_0(m-1))/m),\\
\frac{r^{\,\prime}_0\cdot e\cdot
y}{|y|\exp\left\{\log^{\,\alpha}\left(\frac{er^{\,\prime}_0}{|y|}\right)\right\}}\,,
& y\in B(0, r^{\,\prime}_0)\setminus B(0, (r^{\,\prime}_0(m-1))/m)
\end{array}
\right.
$$
is a family of automorphisms of $B(0, r^{\,\prime}_0),$ while
$\widetilde{g}_m(y)=y$ for $y\in S(0, r^{\,\prime}_0).$ On the other
hand, by construction, there exists a continuous inverse mapping
$\varphi=\pi^{\,-1}$ of $\widetilde{B}(p_0, r_0)$ onto a Euclidean
ball $B(0, r^{\,\prime}_0).$ In this case, we set
$$g_m(p)=(\widetilde{g}_m\circ \varphi)(p), \qquad p\in \widetilde{B}(p_0, r_0)\,.$$
Note that the mappings $g_m$ are almost everywhere differentiable
homeomorphisms of $\widetilde{B}(p_0, r_0)$ onto $B(0,
r^{\,\prime}_0),$ which belong to $ACP^{\,-1}$-class and have $N$-
and $N^{\,-1}$-Luzin properties in ${\Bbb B}^n/G.$ Wherein,
$g_m(p)=\varphi(p)$ for $p\in
\partial\widetilde{B}(p_0, r_0).$
Next we set
$$f_m(p)=\left\{
\begin{array}{rr} (\pi\circ\widetilde{g}_m\circ \varphi)(p), & p\in \widetilde{B}(p_0, r_0)\,,\\
p, &  p\in {\Bbb B}^n/G\setminus \widetilde{B}(p_0, r_0)
\end{array}
\right.\,.$$
The definition of the mappings~$f_m$ implies that $f_m$ are
differentiable almost everywhere homeomorphisms of ${\Bbb B}^n/G$
onto itself, which belong to $ACP^{\,-1}$-class and have $N$- and
$N^{\,-1}$-Luzin properties in ${\Bbb B}^n/G$ (see Figure~\ref{fig3}
for illustration).
\begin{figure}[h]
\centerline{\includegraphics[scale=0.5]{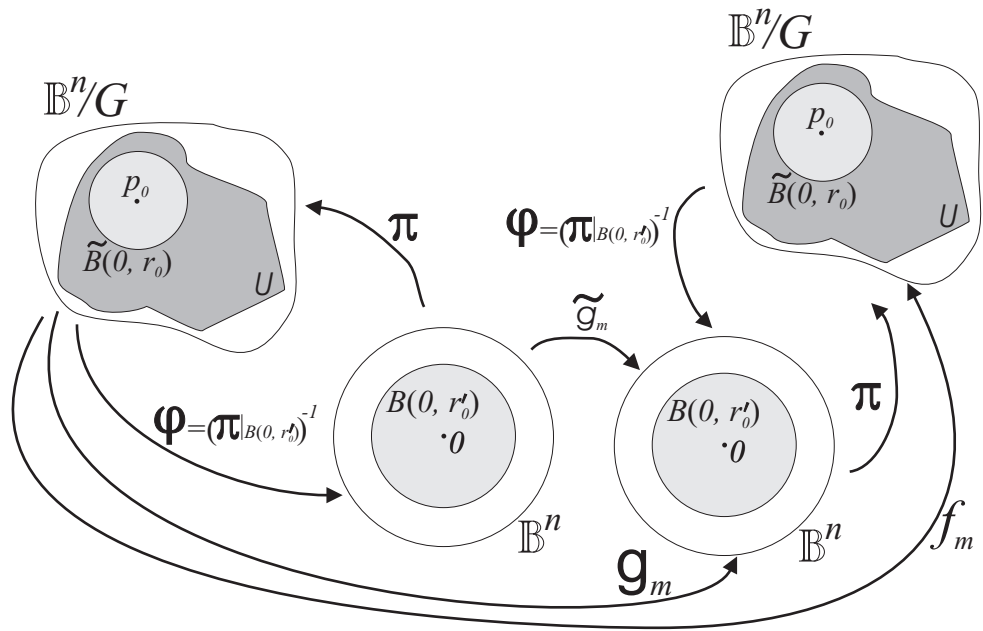}} \caption{The
scheme used to construct the mappings~$f_m$ }\label{fig3}
\end{figure}
We can verify directly that the family of mappings $f_m$ is
equicontinuous at the point~$p_0.$ The last conclusion also follows
from Theorem~\ref{theor4*!}. Notice, that all the conditions of this
theorem are fulfilled, except that the family of mappings $f_m$ acts
into some ball $B_R\subset {\Bbb B}^n/G$ and does not take values
from some continuum $K_m,$ belonging to this ball and having a
diameter no less than $\delta>0,$ $m=1,2,\ldots .$

\medskip
In order for this last condition to be satisfied, we consider the
restriction $f_m|_{\widetilde{B}(p_0, r_0)}$ of this family of
mappings $f_m$ to the ball $\widetilde{B}(p_0, r_0).$ Without loss
of generality, we may assume that the closed ball
$\widetilde{B}(p_0, 2r_0)$ still lies in some normal neighborhood
$U$ of $p_0$ and does not coincide with the whole space ${\Bbb
B}^n/G.$ (Simply, we can initially choose $r_0$ appropriately).
Choose a continuum $K\subset B(0, (2r_0)^{\,\prime})\setminus B(0,
r^{\,\prime}_0)$ in an arbitrary way, where
$(2r_0)^{\,\prime}=\frac{e^{(2r_0)^{\,\prime}}-1}{e^{(2r_0)^{\,\prime}}+1}.$
Then the family $f_m|_{\widetilde{B}(p_0, r_0)}$ satisfies all the
conditions of Theorem~\ref{theor4*!}, in particular, the mappings of
this family do not take on the values of some fixed non-degenerate
continuum~$\pi(K)\subset {\Bbb B}^n/G.$


\medskip
\medskip
{\bf \noindent Evgeny Sevost'yanov} \\
{\bf 1.} Zhytomyr Ivan Franko State University,  \\
40 Bol'shaya Berdichevskaya Str., 10 008  Zhytomyr, UKRAINE \\
{\bf 2.} Institute of Applied Mathematics and Mechanics\\
of NAS of Ukraine, \\
1 Dobrovol'skogo Str., 84 100 Slavyansk,  UKRAINE\\
esevostyanov2009@gmail.com


\begin{thebibliography}{99}

{\small

\bibitem[Ap]{Ap} {\sc Apanasov, B.N.:} Conformal Geometry od Discrete Groups and
Manifolds. - Walter de Gruyter, Berlin, New York, 2000.

\bibitem[Fe]{Fe} {\sc Federer, H.:} Geometric Measure Theory. - Springer, Berlin etc., 1969.

\bibitem[He]{He} {\sc Heinonen, J.:} Lectures on Analysis on metric
spaces. - Springer Science+Business Media: New York, 2001.

\bibitem[IR]{IR} {\sc Ignat'ev~A. and V.~Ryazanov:} Finite mean oscillation in mapping
theory. - Ukr. Mat. Visn. 2:3, 2005, 395--417 (in Russian); English
transl. in Ukr. Math. Bull. 2:3, 2005, 403–-424.

\bibitem[IS]{IS} {\sc Il’yutko, D.P. and E.A.~Sevost’yanov:} Boundary behaviour of open discrete mappings on Riemannian
manifolds. - Sbornik Mathematics 209:5, 2018, 605--651.

\bibitem[Ku]{Ku} {\sc Kuratowski, K.:} Topology, v.~1. -- Academic
Press, New York--London,  1968.

\bibitem[MM]{MM} {\sl Maly~J. and O.~Martio:} Lusin's condition $N$ and mappings of the
class $W_{loc}^{1,n}.$ - J. Reine Angew. Math. 458, 1995, 19--36.

\bibitem[MRV]{MRV} {\sc Martio, O., S. Rickman, and J. V\"{a}is\"{a}l\"{a}:}
Definitions for quasiregular mappings. - Ann. Acad. Sci. Fenn. Ser.
A1 448, 1969, 1--40.

\bibitem[MRSY$_1$]{MRSY$_1$} {\sc Martio, O., V.~Ryazanov, U.~Srebro. and
E.~Yakubov:} Mappings with finite length distortion. -- J. d'Anal.
Math. 93, 2004, 215--236.

\bibitem[MRSY$_2$]{MRSY$_2$} {\sc Martio, O., V. Ryazanov, U. Srebro, and E. Yakubov:} Moduli in modern mapping
theory. - Springer Science + Business Media, LLC, New York, 2009.

\bibitem[MS$_1$]{MS$_1$} {\sc Martio, O., U. Srebro:} Automorphic quasimeromorphic mappings in
${\Bbb R}^n.$ - Acta Math. 135, 1975, 221--247.

\bibitem[MS$_2$]{MS$_2$} {\sc Martio, O., U. Srebro:} Periodic quasimeromorphic
mappings in ${\Bbb R}^n.$ - J. d'Anal. Math. 28:1, 1975, 20--40.

\bibitem[Pol]{Pol} {\sc Poletskii, E.A.:} The modulus method for non-homeomorphic
quasiconformal mappings. - Mat. Sb., 83(2), 1970, 261–-272 (in
Russian).

\bibitem[RR]{RR} {\sc Reimann, H.M. and T.~Rychener:} Funktionen Beschr\"{a}nkter Mittlerer
Oscillation. - Lecture Notes in Math. 487, Springer--Verlag, Berlin
etc., 1975.

\bibitem[Ri]{Ri} {\sc Rickman, S.:} Quasiregular mappings. --
Springer-Verlag, Berlin etc., 1993.

\bibitem[RV$_1$]{RV$_1$}  {\sc Ryazanov, V. and S.~Volkov:} On the Boundary Behavior of Mappings in the Class
$W^{1,1}_{loc}$ on Riemann Surfaces. - Complex Analysis and Operator
Theory 11, 2017, 1503--1520.

\bibitem[RV$_2$]{RV$_2$} {\sc Ryazanov, V. and S.~Volkov:}
Prime ends in the Sobolev mapping theory on Riemann surfaces. - Mat.
Stud. 48, 2017, 24-–36.

\bibitem[Sa]{Sa} {\sc Saks, S.:} Theory of the Integral. - Dover, New York, 1964.

\bibitem[SalSev]{SalSev}{\sc Salimov, R.R. and E.A.~Sevost'yanov: } The Poletskii and V\"{a}is\"{a}l\"{a}
inequalities for the mappings with $(p, q)$-distortion. - Complex
Variables and Elliptic Equations 59:2, 2014, 217--231.

\bibitem[Sev]{Sev}{\sc Sevost'yanov, E.A.:} Local
and boundary behavior of maps in metric spaces. - St. Petersburg
Math. J. 28:6, 2017, 807--824.

\bibitem[SM]{SM} {\sc Sevost'yanov, E.A. and A.A.~Markysh:} On Sokhotski--Casorati--Weierstrass theorem on metric
spaces. - Complex Variables and Elliptic Equations, published online
https://www.tandfonline.com/doi/full/10.1080/17476933.2018.1557155 .

\bibitem[Va$_1$]{Va$_1$} {\sc V\"{a}is\"{a}l\"{a}, J.:} Two new characterizations
for quasiconformality. - Ann. Acad. Sci. Fenn. Ser. A 1 Math. 362,
1965, 1--12.

\bibitem[Va$_2$]{Va$_2$} {\sc V\"{a}is\"{a}l\"{a}, J.:}
Disrete open mappings on manifolds. - Ann. Acad. Sci. Fenn. Ser. A 1
Math. 392, 1966, 1--10.

\bibitem[Va$_3$]{Va$_3$} {\sc V\"{a}is\"{a}l\"{a}, J.:}
Lectures on $n$-dimensional quasiconformal mappings. - Lecture Notes in
Math. 229, Springer-Verlag, Berlin etc., 1971.

\bibitem[Vu]{Vu}  {\sc Vuorinen, M.:} Conformal Geometry and Quasiregular Mappings. -
Lecture Notes in Math. 1319, Springer--Verlag, Berlin etc.,  1988.}


\end{thebibliography}
\end{document}